\documentclass[11pt]{article}
\usepackage{amssymb,amsmath,amsthm,graphicx}
\usepackage{color,enumerate}
\usepackage{authblk}
\usepackage{cite,fullpage}
\usepackage{caption}
\usepackage{subcaption}
\usepackage[ruled,vlined,linesnumbered]{algorithm2e}
\usepackage{relsize}
\usepackage{rotating}

\newtheorem{theorem}{Theorem}[section]
\newtheorem{lemma}[theorem]{Lemma}
\newtheorem{corollary}[theorem]{Corollary}
\newtheorem{proposition}[theorem]{Proposition}

\newtheorem{conjecture}{Conjecture}
\newtheorem*{conjecture1}{Conjecture 1}
\newtheorem*{conjecture2}{Conjecture 2}
\newtheorem*{theoremStrongCliques}{Theorem~\ref{thm:strong-cliques-hardness}}
\newtheorem*{theoremHardness}{Theorem~\ref{thm:hard}}
\newtheorem*{theoremLine}{Theorem~\ref{thm:line}}

\theoremstyle{definition}
\newtheorem{definition}[theorem]{Definition}
\newcommand{\btimes}{\mathbin{\text{\rotatebox[origin=c]{45}{$\boxtimes$}}}}
\newcommand{\ddiamond}{\mathbin{\text{\rotatebox[origin=c]{90}{$\triangleleft\!\!\:\triangleright$}}}}

\newcommand{\Ptriangle}{{\cal T}_{\!\mathsmaller{\triangle}}}
\newcommand{\Pdiamond}{{\cal T}_{\ddiamond}}
\newcommand{\PK}{{\cal T}_{\mathsmaller{\btimes}}}
\newcommand{\Rt}{R_{\!\mathsmaller{\triangle}}}
\newcommand{\Rd}{R_{\ddiamond}}
\newcommand{\RK}{R_{\mathsmaller{\btimes}}}

\usepackage{amssymb}
\usepackage{graphicx}
\begin{document}

\title{Graphs vertex-partitionable into strong cliques}

\author[1,2]{Ademir Hujdurovi\'c\thanks{ademir.hujdurovic@upr.si}}
\author[1,2]{Martin Milani\v c\thanks{martin.milanic@upr.si}}
\author[3]{Bernard Ries\thanks{bernard.ries@unifr.ch}}
\affil[1]{\normalsize University of Primorska, UP IAM, Muzejski trg 2, SI-6000 Koper, Slovenia}
\affil[2]{\normalsize University of Primorska, UP FAMNIT, Glagolja\v ska 8, SI-6000 Koper, Slovenia}
\affil[3]{\normalsize University of Fribourg, Department of Informatics, Bd de P\'erolles 90, CH-1700 Fribourg, Switzerland}

\date{\today}

\maketitle
\begin{abstract}
A graph is said to be {\it well-covered} if all its maximal independent sets are of the same size. In 1999, Yamashita and Kameda introduced a subclass of well-covered graphs, called {\it localizable graphs} and defined as graphs having a partition of the vertex set into strong cliques, where a clique in a graph is {\it strong} if it intersects all maximal independent sets. Yamashita and Kameda observed that all well-covered trees are localizable, pointed out that the converse inclusion fails in general, and asked for a characterization of localizable graphs.

In this paper we obtain several structural and algorithmic results about localizable graphs. Our results include a proof of the fact that every very well-covered graph is localizable and characterizations of localizable graphs within the classes of line graphs, triangle-free graphs, $C_4$-free graphs, and cubic graphs, each leading to a polynomial time recognition algorithm. On the negative side, we prove NP-hardness of recognizing localizable graphs within the classes of weakly chordal graphs, complements of line graphs, and graphs of independence number three. Furthermore, using localizable graphs we disprove a conjecture due to Zaare-Nahandi about $k$-partite well-covered graphs having all maximal cliques of size $k$. Our results unify and generalize several results from the literature.
\end{abstract}


\section{Introduction}

A {\em clique} (resp., {\em independent set}) in a graph is a set of pairwise adjacent (resp., pairwise non-adjacent) vertices.
A clique (resp.,~independent set) in a graph is said to be {\em maximal} if it is not contained in any larger clique (resp., independent set),
and {\em strong} if it intersects every maximal independent set (resp.,~every maximal clique).
A graph is said to be:
\begin{itemize}
  \item {\em well-covered} if all its maximal independent sets have the same size,
  \item {\em co-well-covered} if its complement is well-covered,
  \item {\em localizable} if it admits a partition of its vertex set into strong cliques.
\end{itemize}

The importance of the class of well-covered graphs is partly due to the fact that the maximum independent set problem, which is generally NP-complete, can be solved in polynomial time in the class of well-covered graphs by a greedy algorithm. {Well-covered graphs are also related to the Generalized Kayles game, a two-person game played on a graph, in which two players alternate removing a vertex and all its neighbors and the player who last removes a vertex wins, see~\cite{MR737090}. For well-covered graphs (and, more generally, for parity graphs~\cite{MR1352783}), the outcome of the game is independent of how the players move. Well-covered graphs also play an important role in commutative algebra, where they are typically referred to as {\it unmixed} graphs, see, e.g.,~\cite{MR1031197,MR2585907,MR2231097,MR3177496,MR3272076,MR2932582,MR3339410}. The well-coveredness property of a graph is equivalent to
the property that the simplicial complex of the independent sets of $G$ is {\it pure} and generalizes the algebraically defined concept of
a Cohen-Macaulay graph (see, e.g.,~\cite{MR3512661}). Fur further background on well-covered graphs, we refer to the surveys by Plummer~\cite{MR1254158} and Hartnell~\cite{MR1677797}.

Localizable graphs form a subclass of well-covered graphs. They were introduced by Yamashita and Kameda in 1999, in the concluding remarks of their paper~\cite{MR1715546}. This property appeared implicitly in other places in the literature, for example in~\cite[Proposition 1]{MR737090} and in~\cite[Theorem 2.1]{MR3356635}. In their 1983 paper~\cite{MR737090}, Finbow and Hartnell proved that every well-covered graph of girth at least $8$ has a perfect matching formed by pendant edges. Since every pendant edge is a strong clique, this implies that every well-covered graph of girth at least $8$ is localizable.} In~1999~\cite{MR1715546}, Yamashita and Kameda observed that all well-covered trees are localizable, pointed out that the converse inclusion fails in general, and asked for a characterization of localizable graphs.

Motivated by this question, we initiate in this work the study of localizable graphs.
Our results can be divided into three main parts:
\begin{enumerate}[1)]
  \item hardness results related to the problem of recognizing localizable graphs,
  \item characterizations of localizable graphs within the classes of line graphs, triangle-free graphs, $C_4$-free graphs, and cubic graphs, and
  \item counterexamples to a conjecture by Zaare-Nahandi about $k$-partite well-covered graphs having all maximal cliques of size $k$.
\end{enumerate}

\medskip
We summarize our results, approach, and related work as follows.
\bigskip

\noindent{\bf 1. Hardness results.}

\medskip
\noindent We show that recognizing localizable graphs is hard even for rather restricted graph classes:

\begin{theorem}\label{thm:hard}
The problem of recognizing localizable graphs is:
\begin{itemize}
  \item co-NP-complete in the class of weakly chordal graphs (and consequently in the class of perfect graphs),
  \item NP-complete in the class of complements of line graphs of triangle-free $k$-regular graphs, for every fixed $k\ge 3$,
  \item NP-complete in the class of graphs of independence number $k$, for every fixed $k\ge 3$.
\end{itemize}
\end{theorem}

We obtain these hardness results by employing two different techniques, both based on the fact that a graph is localizable if and only if it is well-covered and semi-perfect. Following Zaare-Nahandi~\cite{MR3356635}, we say that a graph $G$ is {\it semi-perfect} if $\theta(G)= \alpha(G)$, where $\theta(G)$ and $\alpha(G)$ denote the clique cover number of $G$ and its independence number, respectively.
The NP-hardness of recognizing localizable graphs within the class of weakly chordal graphs follows from the corresponding hardness result for well-covered graphs~\cite{MR1161178,MR1217991}, using the fact that every weakly chordal graph is semi-perfect.
The remaining two NP-hardness results hold for graph classes where well-coveredness can be tested in polynomial time and
are obtained using reductions from edge- and vertex-colorability problems, by identifying two classes of well-covered graphs in which testing semi-perfection is intractable.

We further elaborate on the approach for proving the hardness result for the class of weakly chordal graphs and derive new hardness results for two natural problems related to strong cliques:

 \begin{theorem}\label{thm:strong-cliques-hardness}
The following problems are NP-hard, even for weakly chordal graphs:
\begin{enumerate}
  \item The problem of determining whether, given a graph and a partition of its vertex set into cliques, each clique in the partition is strong. (This problem is also co-NP-complete.)
  \item The problem of determining whether every vertex of a given graph is contained in a strong clique.
\end{enumerate}
\end{theorem}

In particular, while it is not known whether the problem of recognizing localizable graphs is in NP, we show that the most natural certificate for a yes instance (that is, a partition of the graph's vertex set into strong cliques) is most likely not verifiable in polynomial time. Our approach also gives an alternative proof of the fact that the problem of testing whether a given clique in a graph is strong is co-NP-complete, as shown by Zang~\cite{MR1344757}.
Note that the computational complexity status of several related problems is open. This is the case for the problem of recognizing graphs every induced subgraph of which has a strong clique, graphs in which every edge is contained in a strong clique, or graphs in which every maximal clique is strong (see the discussion following Corollary~\ref{cor:weakly-chordal} in Section~\ref{sec:weakly-chordal} for details).

\bigskip

\noindent{\bf 2. Characterizations.}

\medskip
\noindent Since a semi-perfect graph is localizable if and only if it is well-covered, any result characterizing well-covered graphs within a class of semi-perfect graphs immediately implies the same characterization of localizable graphs within the class. This yields characterizations of localizable graphs within the classes of bipartite graphs~\cite{MR0469831},
or, more generally, of triangle-free semi-perfect graphs~\cite{MR1264476}, as well as
of chordal graphs~\cite{MR1368737}, or, more generally, of $C_4$-free semi-perfect graphs~\cite{MR1264476}.
As our second set of results, we extend the list of graph classes in which localizable graphs are characterized, by characterizing
localizable graphs within the classes of triangle-free graphs, $C_4$-free graphs, cubic graphs, and line graphs. Our characterizations also imply polynomial time recognition algorithms of localizable graphs within these classes.

The most involved of these characterizations is the one for line graphs. Recall that the line graph of a graph $G$ is well-covered if and only if $G$ is {\it equimatchable}, that is, if all maximal matchings of $G$ are of the same size. The literature on equimatchable graphs is extensive, see, e.g.,~\cite{MR777180,Favaron,Meng,KPS2003,KP2009,Lewin,demange_ekim_equi,FHV2010,EK2016}. In particular, several characterizations and polynomial time recognition algorithms of equimatchable graphs are known~\cite{Lewin,MR777180,demange_ekim_equi}. Our characterization of localizable line graphs is derived independently of the characterizations of equimatchable graphs and implies, in particular, that the equimatchable bipartite graphs
are the only triangle-free graphs whose line graphs are localizable.

The characterization of localizable graphs within the class of triangle-free graphs is obtained using known characterizations of very well-covered graphs~\cite{Staples,MR677051} and the fact that every very well-covered graph is localizable (which we show in Section~\ref{sec:very-well-covered}).
The characterization of localizable $C_4$-free graphs generalizes the above-mentioned results on well-covered chordal~\cite{MR1368737}, resp., $C_4$-free semi-perfect graphs~\cite{MR1264476}. We obtain this result by first characterizing strong cliques in $C_4$-free graphs. To put these results in perspective, note that no characterization of well-covered triangle-free or $C_4$-free graphs is known. (Well-covered graphs of girth at least five were characterized by Finbow et al.~\cite{MR1198396} and well-covered graphs without a (not necessarily induced) subgraph isomorphic to $C_4$ were studied by Brown et al.~\cite{MR2340625}.)

The class of localizable cubic graphs consists of an infinite family of planar cubic graphs
along with three small graphs ($K_4$, $K_{3,3}$, and the complement of $C_6$). This classification can be obtained from the classification of well-covered cubic graphs due to Campbell et al.~\cite{MR1220613}. However, the way to characterizing the well-covered cubic graphs was long, building on earlier results due to Campbell and Plummer~\cite{Campbell,MR942505}; we give a short direct proof of the classification of localizable cubic graphs.

We postpone the exact statements of the characterizations to the respective sections (Theorems~\ref{thm:triangle-free} and~\ref{thm:C4-free} in Section~\ref{sec:triangle-free}, Theorem~\ref{thm:3-regular} in Section~\ref{sec:cubic}, and Theorem~\ref{thm:line} in Section~\ref{sec:line}).

\bigskip
\noindent{\bf 3. Counterexamples to a conjecture by Zaare-Nahandi.}

\medskip
\noindent We give a family of counterexamples to the following recent conjecture closely related to localizable graphs.

\begin{conjecture}[Zaare-Nahandi~\cite{MR3356635}]\label{conj:Zaare-Nahandi}
Let $G$ be an $s$-partite well-covered graph in which all maximal cliques are of size $s$. Then $G$ is semi-perfect.
\end{conjecture}

In terms of localizable graphs, the conjecture can be equivalently posed as follows:

\begin{conjecture}\label{conj2}
Let $G$ be a localizable co-well-covered graph.
Then $\overline{G}$ is localizable.
\end{conjecture}

We disprove these two equivalent conjectures by constructing an infinite family of counterexamples to the weaker statement saying that every localizable co-well-covered graph has a strong independent set. We also give a related hardness result showing that it is NP-hard to determine whether the complement of a given localizable co-well-covered graph is localizable. The proof is based on a reduction from the $3$-colorability problem in triangle-free graphs and shows a way how to transform, in a simple way, any triangle-free graph of chromatic number more than $3$ to a counterexample to Conjecture~\ref{conj:Zaare-Nahandi}.

\bigskip
\noindent{\bf Structure of the paper.}
In Section~\ref{sec:preliminaries} we give several equivalent formulations of the property of localizability.
In Section~\ref{sec:very-well-covered} we show that every very well-covered graph is localizable.
In Section~\ref{sec:hard} we develop the hardness results, including Theorems~\ref{thm:hard} and~\ref{thm:strong-cliques-hardness}.
Section~\ref{sec:poly} is devoted to establishing characterizations of localizable graphs within the classes of triangle-free graphs, $C_4$-free graphs, cubic graphs, and line graphs. Counterexamples to Conjecture~\ref{conj:Zaare-Nahandi} are given in Section~\ref{sec:counterexample}.
We conclude the paper with some open questions in Section~\ref{sec:conclusion}.

\section{Equivalent formulations of localizability}\label{sec:preliminaries}

In this section, we give several equivalent formulations of the property of localizability. First we recall some definitions and fix some notation. We consider only finite, simple and undirected graphs. Given a graph $G=(V,E)$, its complement $\overline{G}$ is the graph with vertex set $V$ in which two distinct vertices are adjacent if and only if they are non-adjacent in $G$. By $K_n$, $P_n$, and $C_n$ we denote the $n$-vertex complete graph, path, and cycle, respectively, and by $K_{m,n}$ the complete bipartite graph with parts of sizes $m$ and $n$. The {\it degree} of a vertex $v$ in a graph $G$ is denoted by $d_G(v)$, its neighborhood by $N_G(v)$ (or simply by $N(v)$ if the graph is clear from the context), and its closed neighborhood by $N_G[v]$ (or simply by $N[v]$).
For a set of vertices $X\subseteq V(G)$, we denote by $N_G(X)$ (or $N(X)$) the set of all vertices in $V(G)\setminus X$ having a neighbor in $X$.
A {\it triangle} in a graph is a clique of size $3$; a graph is {\it triangle-free} if it has no triangles. Similarly, a graph is {\it $C_4$-free} if it has no induced subgraph isomorphic to a $C_4$. We will often identify a triangle with the set of its edges; whether we consider a triangle as a set of vertices or as a set of edges will always be clear from the context.

Given a graph $G$, we denote by $\alpha(G)$ its {\it independence number}, that is, the maximum size of an independent set in $G$,
by $i(G)$ its {\it independent domination number}, that is, the minimum size of an independent dominating set in $G$
(equivalently: the minimum size of a maximal independent set in $G$), by $\omega(G)$ its {\it clique number}, that is, the maximum size of a clique in $G$, by $\chi(G)$ its {\it chromatic number}, that is, the minimum number of independent sets that partition its vertex set,
and by $\theta(G)$ its {\it clique cover number}, that is,
the minimum number of cliques that partition its vertex set.
Every graph $G$ has
$\alpha(G) = \omega(\overline{G})$,
$\theta(G) = \chi(\overline{G})$, $\chi(G)\ge \omega(G)$, and
$\theta(G)\ge \alpha(G)$.
It follows that every graph $G$ satisfies the following chain of inequalities:
\begin{equation}\label{ineq}
i(G)\le \alpha(G)\le \theta(G)\,.
\end{equation}
Clearly, a graph $G$ is well-covered if and only if $i(G) = \alpha(G)$.

For a positive integer $k$, we say that a graph $G$ is {\em $k$-localizable} if it admits a partition of its vertex set into exactly $k$ strong cliques. {Recall that a graph $G$ is said to be {\it semi-perfect} if $\theta(G)= \alpha(G)$, that is, if there exists a collection of $\alpha(G)$ cliques partitioning its vertex set. We will refer to such a collection as an {\it $\alpha$-clique cover} of $G$. Thus, $G$ is semi-perfect if and only if it has an $\alpha$-clique cover, and $G$ is $\alpha(G)$-localizable if and only if it has a $\alpha$-clique cover in which every clique is strong.

Now we have everything ready to prove the following equivalent formulations of localizability.

\begin{proposition}\label{prop:alpha-chi-bar}
For every graph $G$, the following statements are equivalent.
\begin{enumerate}
  \item[(a)] $G$ is localizable.
  \item[(b)] $G$ is $\alpha(G)$-localizable (equivalently, $G$ has an $\alpha$-clique cover in which every clique is strong).
  \item[(c)] $G$ has an $\alpha$-clique cover and every clique in every $\alpha$-clique cover of $G$ is strong.
  \item[(d)] $G$ is well-covered and semi-perfect.
  \item[(e)] $i(G) = \theta(G)$.
\end{enumerate}
\end{proposition}

\begin{proof}
(a) $\Rightarrow$ (b): Suppose that $G$ is localizable, and let $C_1,\ldots, C_k$ be a collection of strong cliques of $G$ partitioning its vertex set.
Let $S$ be a maximal independent set of $G$. Then $S$ intersects each $C_i$ in a vertex, which implies that
$|S| = \sum_{i = 1}^k |C_i\cap S| = k$. Since $S$ was arbitrary, $G$ is well-covered, with
$\alpha(G) = k$. In particular, $G$ is $\alpha(G)$-localizable. Thus, (a) implies (b).

(b) $\Rightarrow$ (d):
Suppose that $G$ is $\alpha(G)$-localizable. Then $\theta(G)\le \alpha(G)$. Since the opposite inequality holds for every graph, we conclude that
$\theta(G)= \alpha(G)$. Moreover, since $V(G)$ has a partition into $\alpha(G)$ strong cliques, every maximal independent set in $G$ is of size $\alpha(G)$.
Thus, $G$ is also well-covered and the implication (b) $\Rightarrow$ (d) follows.

(d) $\Rightarrow$ (c): Suppose that $G$ is well-covered and semi-perfect.
Since $G$ is semi-perfect, it has an $\alpha$-clique cover. Now consider and arbitrary $\alpha$-clique cover $C_1,\ldots, C_{\alpha(G)}$ of $G$.
We will show that each clique $C_i$ is strong. Suppose this is not the case. Without loss of generality, assume that $C_1$ is not strong.
Then, there exists a maximal independent set $S$ of $G$ disjoint from $C_1$.
Consequently, $|S| = \sum_{i = 2}^{\alpha(G)}|S\cap C_i|\le \alpha(G)-1$, contrary to the fact that $G$ is well-covered.
This proves that every $C_i$ is strong and establishes the implication (d) $\Rightarrow$ (c).

(c) $\Rightarrow$ (a): Trivial.

(d) $\Leftrightarrow$ (e): Recall that $G$ is well-covered if and only if $i(G) = \alpha(G)$. Therefore, condition (d) is equivalent to
$i(G) = \alpha(G) = \theta(G)$, which, by~\eqref{ineq}, is equivalent to condition (e).

The above equalities and implications complete the proof of the proposition.
\end{proof}

Proposition~\ref{prop:alpha-chi-bar} will be applied several times in the paper.
At this point let us mention some of its consequences.

\begin{corollary}\label{cor:semi-perfect}
For every semi-perfect graph $G$, the following statements are equivalent.
\begin{enumerate}
  \item $G$ is localizable.
  \item $G$ is $\alpha(G)$-localizable (equivalently, $G$ has an $\alpha$-clique cover in which every clique is strong).
  \item $G$ is well-covered.
  \item Every clique in every $\alpha$-clique cover of $G$ is strong.
\end{enumerate}
\end{corollary}

Corollary~\ref{cor:semi-perfect} implies a characterization of well-covered graphs within a special class of tripartite graphs given by Haghighi~\cite[Theorem 3.2]{MR3292063} as well as a generalization to special $r$-partite graphs due to Jafarpour-Golzari and Zaare-Nahandi~\cite[Theorem 2.3]{unmixed2015}.\footnote{Indeed, Theorem 3.2 of~\cite{MR3292063} considers a tripartite graph $G$ given with an $\alpha$-clique cover into $\alpha(G)$ triangles; the conditions stated in the theorem are equivalent to the property that each triangle in this particular $\alpha$-clique cover is strong.
The situation is similar with Theorem 2.3 from~\cite{unmixed2015}, except that triangles are replaced with $r$-cliques.}
The equivalence between conditions 2, 3, and 4 of Corollary~\ref{cor:semi-perfect} also follows from a result of Zaare-Nahandi, Theorem 2.1 in~\cite{MR3356635}.

A graph $G$ is {\it perfect} if $\chi(H) = \omega(H)$ holds for every induced subgraph $H$ of $G$. Since the complementary graph $\overline{G}$ is perfect whenever $G$ is perfect~\cite{MR0302480}, every perfect graph $G$ is also semi-perfect. In particular, Corollary~\ref{cor:semi-perfect} implies:

\begin{corollary}\label{cor:perfect}
A perfect graph is well-covered if and only if it is localizable.
\end{corollary}}

\section{Very well-covered graphs are localizable}\label{sec:very-well-covered}

A well-covered graph $G$ is said to be {\em very well-covered} if it has no isolated vertices and $\alpha(G) = |V(G)|/2$.
{Very well-covered graphs were studied in commutative algebra, in the context of connections between properties of a graph $G$ with properties of
the simplicial complex whose faces are the independent sets of $G$, and of the edge ideal of $G$, see, e.g.,~\cite{MR2793950,MR3512661}.}
Very well-covered graphs were characterized by Staples in 1975~\cite[Theorem 1.11]{Staples} and independently by Favaron in 1982~\cite{MR677051},
as follows. Given a matching $M$ in a graph $G$, we say that $M$ {\it satisfies property (P)} if
for every edge $uv\in M$, we have $N(u)\cap N(v) = \emptyset$ and
every vertex of $N(u)$ is adjacent to every vertex of $N(v)$.

\begin{theorem}[Staples~\cite{Staples} and Favaron~\cite{MR677051}]\label{thm:Favaron-original}
For a graph $G$ without isolated vertices, the following properties are equivalent:
\begin{enumerate}
  \item $G$ is very well-covered.
  \item There exists a perfect matching in $G$ satisfying property (P).
  \item There exists a perfect matching in $G$ and every perfect matching of $G$ satisfies (P).
\end{enumerate}
\end{theorem}

Restricted to graphs without isolated vertices, condition 2 in the above theorem has already appeared as a characterizing condition of well-covered bipartite graphs (by Ravindra in 1977~\cite{MR0469831}) and, more generally, of well-covered triangle-free semi-perfect graphs (by Dean and Zito in 1994~\cite{MR1264476}).

The following lemma will enable us to state the result of Theorem~\ref{thm:Favaron-original} more succinctly in terms of strong cliques.

\begin{lemma}\label{lem:strong}
Let $G$ be a graph and let $uv\in E(G)$. Then, $\{u,v\}$ is a strong clique in $G$ if and only if
$uv$ is not contained in any triangle and every vertex of $N(u)$ is adjacent to every vertex of $N(v)$.
\end{lemma}

\begin{proof}
Suppose first that $\{u,v\}$ is a strong clique in $G$. Then,
$\{u,v\}$ is a maximal clique and therefore it is not contained in any triangle.
Suppose for a contradiction that some neighbor of $u$, say $x$, is not adjacent to some neighbor of $v$, say $y$.
Then $x\neq v$ and $y\neq u$; however, extending the set $\{x,y\}$ to a maximal independent set of $G$
results in a maximal independent set disjoint from $\{u,v\}$, contradicting the assumption that $\{u,v\}$ is strong.

Conversely, suppose that $uv$ is not contained in any triangle and every vertex of $N(u)$ is adjacent to every vertex of $N(v)$. If $\{u,v\}$ is not strong, then $G$ has a maximal independent set, say $I$,
disjoint from $\{u,v\}$. Since $I$ is maximal, each of $u$ and $v$ have a neighbor in $I$.
Since $uv$ is not contained in any triangle, $u$ and $v$ do not have any common neighbors; therefore
$I$ contains a pair of distinct vertices $x$ and $y$ such that $ux\in E(G)$ and $vy\in E(G)$.
Since $x\in N(u)$ and $y\in N(v)$ are non-adjacent, this contradicts the assumption that
every vertex of $N(u)$ is adjacent to every vertex of $N(v)$.
\end{proof}

Lemma~\ref{lem:strong} implies that a matching $M$ in a graph $G$ satisfies property (P) if and only if every edge of $M$ is a strong clique.
Therefore, Theorem~\ref{thm:Favaron-original} can be equivalently stated as follows.

\begin{theorem}\label{thm:Favaron}
For a graph $G$ without isolated vertices, the following properties are equivalent:
\begin{enumerate}
  \item $G$ is very well-covered.
  \item There exists a perfect matching in $G$ every edge of which is a strong clique.
  \item There exists a perfect matching in $G$, and every edge of every perfect matching in $G$ is a strong clique.
  \item $G$ is localizable, with a partition of its vertex set into strong cliques of size two.
\end{enumerate}
In particular, every very well-covered graph is localizable.
\end{theorem}

\section{Hardness results}\label{sec:hard}

The problem of recognizing localizable graphs is the following decision problem: ``Given a graph $G$, is $G$ localizable?'' In this section, we give several hardness proofs for this problem in particular graph classes.

\subsection{On a hardness proof of recognizing well-covered graphs and its implications for problems related to strong cliques}\label{sec:weakly-chordal}

It follows from results of Prisner et al.~\cite{MR1368737} that localizable graphs can be recognized in polynomial time within the class of chordal graphs. A graph is {\it weakly chordal} if neither $G$ nor its complement contain an induced cycle of length at least~$5$. Weakly chordal graphs form an important subclass of perfect graphs generalizing the class of chordal graphs. Sankaranarayana and Stewart~\cite{MR1161178} and Chv\'atal and Slater~\cite{MR1217991} proved the following.

\begin{theorem}\label{thm:wc-rec-NP-hard}
The problem of recognizing well-covered graphs is co-NP-complete, even for weakly chordal graphs.
\end{theorem}

Since weakly chordal graphs are perfect, Corollary~\ref{cor:perfect} implies that
a weakly chordal graph is well-covered if and only if it is localizable. We therefore obtain:

\begin{theorem}\label{thm:weakly-chordal}
The problem of determining whether a given weakly chordal graph is localizable is co-NP-complete.
\end{theorem}

Both proofs of Theorem~\ref{thm:wc-rec-NP-hard} from~\cite{MR1161178,MR1217991} are based on a reduction from the $3$-SAT problem.
The input to the $3$-SAT problem is a set $X = \{x_1, \ldots, x_n\}$ of Boolean variables and a collection ${\cal C} = \{C_1,\ldots, C_m\}$ of clauses of length $3$ over $X$. Each clause is a disjunction of exactly three {\it literals}, where a literal is either one of the variables (say $x_i$) or its negation (denoted by $\overline{x_i}$). The task is to determine whether the formula $\varphi = \bigwedge_{i = 1}^mC_i$ is satisfiable, that is, whether there is a truth assignment to the $n$ variables which makes all the clauses simultaneously evaluate to true.
The reduction produces from a given $3$-SAT instance $I$ a weakly chordal graph $G(I)$ such that $G(I)$ is not well-covered if and only if the formula is satisfiable.

We now show that the above approach can be used to derive new hardness results for two natural problems related to strong cliques. The first of these results implies that the obvious certificate for yes instances of the problem of recognizing localizable graphs is most likely not verifiable in polynomial time.

\begin{theoremStrongCliques}[restated]
The following problems are NP-hard, even for weakly chordal graphs:
\begin{enumerate}
  \item The problem of determining whether, given a graph and a partition of its vertex set into cliques, each clique in the partition is strong.
  (This problem is also co-NP-complete.)
  \item The problem of determining whether every vertex of a given graph is contained in a strong clique.
\end{enumerate}
\end{theoremStrongCliques}

\begin{proof}
Let us first recall the reduction used in the proofs of Theorem~\ref{thm:wc-rec-NP-hard} from~\cite{MR1161178,MR1217991}.
Given an input $(X,{\cal C})$ to the $3$-SAT problem as above, representing a formula $\varphi$,
a graph $G = G(X,{\cal C})$ is constructed with vertex set $V(G) = C\cup L$ where $C = \{c_1,\ldots, c_m\}$ forms a clique, $L = \{x_1,\overline{x_1}, \ldots, x_n,\overline{x_n}\}$,
for each $i\in \{1,\ldots, n\}$, vertices $x_i$ and $\overline{x_i}$ are adjacent,
for each $c_i\in C$ and each literal $\ell\in L$, $\{c_i,\ell\}\in E(G)$ if and only if clause $C_i$ contains literal $\ell$, and there are no other edges. Then, $\varphi$ is satisfiable if and only if $G$ is not well-covered~\cite{MR1161178,MR1217991}.

We use a slight modification of the above reduction. First, note that the $3$-SAT problem remains NP-complete on instances in which no clause contains a pair of the form $\{x_i,\overline{x_i}\}$. Consequently, we may assume that none of the cliques $\{x_i,\overline{x_i}\}$ in $G$ is contained in the neighborhood of a vertex in $C$. Since $N(\{x_i,\overline{x_i}\})$ is a clique, it follows that each clique $\{x_i,\overline{x_i}\}$ is strong in $G$.

We claim that the formula $\varphi$ is satisfiable if and only if the clique $C$ is not strong.
Indeed, on the one hand, literals in $L$ that are set to true in a satisfying assignment form a maximal independent set disjoint from $C$.
Conversely, if $C$ is not strong then any maximal independent set $I$ disjoint from $C$ necessarily consists of exactly one literal from each
pair $\{x_i,\overline{x_i}\}$; setting all literals in $I$ to true yields a satisfying assignment.

It follows that each of the cliques in the partition $\{C,\{x_1,\overline{x_1}\}, \ldots, \{x_1,\overline{x_n}\}\}$ of $V(G)$
is strong if and only if $C$ is strong, if and only if the formula is not satisfiable. Consequently, the problem of determining whether, given a graph and a partition of its vertex set into cliques, each clique in the partition is strong, is NP-hard. The co-NP-completeness of this problem follows from the observation that a short certificate of a no instance consists of a pair $(K,I)$ where $K$ is one of the cliques in the partition and $I$ is a maximal independent set disjoint from $K$.

We now prove the second statement of the theorem, that is, that the problem of determining whether every vertex of a given weakly chordal graph is contained in a strong clique is NP-hard.
We use another slight modification of the above reduction.
Observe that each vertex $\ell\in L$ is contained in a strong clique (namely the $K_2$ containing the literal $\ell$
and its negation). Moreover, we will now show that we may assume that no clique of $G$
containing both a vertex from $L$ and a vertex from $C$ is strong.
To this end, note that the $3$-SAT problem remains NP-complete on instances in which, in addition to the property that no clause contains a pair of the form $\{x_i,\overline{x_i}\}$ (as above), no variable $x_i$ is such that every clause contains either $x_i$ or $\overline{x_i}$.
Indeed, instances having a variable $x_i$ such that every clause contains either $x_i$ or $\overline{x_i}$ can be solved in polynomial time,
by solving two instances of the (polynomially solvable) $2$-SAT problem corresponding to setting $x_i$ to true (resp., to false) and obtained by
eliminating all clauses containing $x_i$ (resp., $\overline{x_i}$) and deleting $\overline{x_i}$ (resp., $x_i$) from all clauses containing it.
Now, let $K$ be a clique of $G$ containing both a vertex from $L$ and a vertex from $C$, say $\ell\in L\cap K$ and $c_i\in C\cap K$.
Let us denote by $\overline{\ell}$ the literal complementary to $\ell$.
It follows from the above that $\{\ell, \overline{\ell}\}$ is strong and in particular that $\overline{\ell}$ is not in $K$.
Therefore $K\cap L = \{\ell\}$. By the above assumption on the $3$-SAT instance, there is a clause $c_j$ that contains neither $\ell$ not $\overline{\ell}$.
Then $c_j\neq c_i$ and the set $I = \{\overline{\ell}, c_j\}$ is an independent set of $G$ such that every vertex of $K$ has a neighbor in $I$.
Hence, any maximal independent set $I'$ of $G$ with $I\subseteq I'$ is disjoint from $K$, hence $K$ is not strong.

We know that every vertex $\ell\in L$ is contained in a strong clique (namely the $K_2$ containing the literal $\ell$
and its negation), and, since no clique of $G$ containing both a vertex from $L$ and a vertex from $C$ is strong,
a vertex in $C$ is contained in a strong clique if and only if $C$ is strong.
It follows that every vertex of $G$ is contained in a strong clique if and only if $C$ is strong, which, as argued above, is
an NP-hard problem.
\end{proof}

Since a graph $G$ is weakly chordal if and only if its complement $\overline{G}$ is weakly chordal, Theorem~\ref{thm:hard}
has the following consequence.

\begin{corollary}\label{cor:weakly-chordal}
The following problems are NP-hard, even for weakly chordal graphs:
\begin{enumerate}
  \item The problem of determining whether, given a graph and a partition of its vertex set into independent sets, each independent set in the partition is strong. (This problem is also co-NP-complete.)
  \item The problem of determining whether every vertex of a given graph is contained in a strong independent set.
\end{enumerate}
\end{corollary}

To put the results of Theorem~\ref{thm:strong-cliques-hardness} and its corollary in a broader context, let us recall the following facts related to the theorem and its proof:
 \begin{itemize}
\item The proof of Theorem~\ref{thm:strong-cliques-hardness} given above also gives an alternative proof of the fact that the problem of testing whether a given clique in a graph is strong is co-NP-complete, as shown by Zang~\cite{MR1344757}.
\item A graph $G$ is said to be {\em very strongly perfect} if in every induced subgraph of $G$, each vertex belongs to a strong independent set.
The fact that the problem of determining whether every vertex of a given graph is contained in a strong independent set is NP-hard
contrasts with the fact that the class of very strongly perfect graphs can be recognized in polynomial time. This follows from the
result of Ho{\`a}ng~\cite{MR888682} showing that the class of very strongly perfect graphs coincides
with the class of {\it Meyniel} graphs~\cite{MR0439682} (defined as graphs in which every odd cycle of length at least $5$
has at least two chords) and a polynomial time recognition algorithm for the class of
Meyniel graphs due to Burlet and Fonlupt~\cite{MR778765}.

    \item Another related problem is that of determining whether a given graph contains a strong clique. As shown by Ho{\`a}ng, this problem is NP-hard~\cite{MR1301855}.
   \item The complexity status of each of the following three problems is unknown:
   (i) the problem of recognizing {\em strongly perfect graphs}~\cite{MR715895,MR778749}, introduced by Berge as graphs every induced subgraph of which has a strong independent set,
   (ii) the problem of determining whether {\it every} edge of a given graph is contained in a strong clique or, equivalently, the problem of recognizing {\it general partition graphs} (see, e.g.,~\cite{MR1212874,MR2794315,MR2080087}), and
   (iii) the problem of recognizing {\em CIS graphs}, defined as graphs in which {\it every} maximal clique is strong, or, equivalently, as graphs in which every maximal independent set is strong (see, e.g.,~\cite{MR2489416,MR2496915,MR2755907,MR3141630,MR3278773}).
\end{itemize}

\subsection{Recognizing localizable graphs with small independence number}\label{sec:k-loc}

Graphs appearing in the hardness proofs of Section~\ref{sec:weakly-chordal} may have arbitrarily large
independence number. We now show that the problem of recognizing localizable graphs is hard already for graphs of independence number $3$ (and more generally, for graphs of independence number $k$, for every fixed $k\ge 3$).
Recall that a graph is $k$-localizable if its vertex set can be partitioned into $k$ strong cliques, and that, by Proposition~\ref{prop:alpha-chi-bar}, a graph is localizable if and only if it is $k$-localizable for $k = \alpha(G)$.

\begin{theorem}\label{prop:k-localizable}
The problem of recognizing $k$-localizable graphs is polynomially solvable for \hbox{$k\in \{1,2\}$}, and NP-complete for all $k\ge 3$.
\end{theorem}

\begin{proof}
Clearly, a graph is $1$-localizable if and only if it is complete. Also, it is easy to see that a graph $G$ is $2$-localizable if and only if $\overline{G}$ is a bipartite graph without isolated vertices. Therefore, $k$-localizable graphs can be recognized in polynomial time for $k\le 2$.

Suppose now that $k\ge 3$. To show membership of the problem in NP, observe that for fixed $k$, given a partition of the vertex set of a graph into $k$ cliques, we can test in polynomial time whether each of these cliques is strong. Indeed, the existence of such a partition implies that any independent set has at most $k$ vertices, so we can enumerate all the maximal independent sets of $G$ in polynomial time, and check for each of them if it intersects each of the cliques in the partition (equivalently, if it is of size $k$).

To show hardness, we make a reduction from the NP-hard $k$-{\sc Colorability} problem: ``Given a graph $G$, is $\chi(G)\le k$?''.
Let $G$ be an input graph to $k$-{\sc Colorability}. We may assume that $\omega(G)\le k$, since otherwise $G$ is not $k$-colorable (as $k$ is fixed, this condition can be tested in polynomial time). We construct a new graph $G'$ (containing $G$ as induced subgraph) by adding, for every maximal clique $C$ of $G$, a clique of $k-|C|$ new vertices and make each of these new vertices adjacent to all vertices of $C$.
By construction, every maximal clique in $G'$ is of size $k$.

To complete the proof, we will show that $G$ is $k$-colorable if and only if the complement of $G'$ is $k$-localizable.
Suppose that $G$ is $k$-colorable, and let $c$ be a $k$-coloring of $G$.
Extend $c$ to a $k$-coloring $c'$ of $G'$. The color classes of $c'$ define a partition of the vertex set
of the complement of $G'$ into cliques $C_1,\ldots, C_k$. Since all maximal cliques of $G'$ are of size $k$, every color class contains a vertex
of each maximal clique, which means that each $C_i$ is a strong clique in the complement of $G'$.
Conversely, if the complement of $G'$ admits a partition of its vertex set into strong cliques, say, $C_1,\ldots, C_k$,
then $G'$ is $k$-colorable, and so is $G$, as an induced subgraph of $G'$.
\end{proof}

Proposition~\ref{prop:alpha-chi-bar} and Theorem~\ref{prop:k-localizable} have the following consequence.

\begin{corollary}\label{cor:alpha}
For every $k\ge 3$, the problem of determining whether a given graph $G$ with $\alpha(G)=k$ is localizable is NP-complete.
\end{corollary}

Note that the graphs in the above reduction for which ($k$-)localizability is tested are well-covered.

To put the result of Corollary~\ref{cor:alpha} in perspective, observe that for every fixed $k$, testing if a given graph $G$ with $\alpha(G)\le k$ is well-covered can be done in polynomial time by enumerating all the $O(|V(G)|^k)$ independent sets
and comparing any pair of maximal ones with respect to their cardinality.

\subsection{Complements of line graphs of triangle-free graphs}

We continue by pointing out a connection between localizable graphs and edge colorings, which implies that
the problem of recognizing localizable graphs is hard also for complements of line graphs of triangle-free graphs.
First we recall some definitions and notation. The {\it line graph} of a graph $G$ is the graph $L(G)$ with vertex set $E(G)$, in which two distinct vertices are adjacent if and only if they have a common endpoint as edges of $G$. The minimum (resp.~maximum) degree of a vertex in a graph $G$ is denoted by $\delta(G)$ (resp.,~$\Delta(G)$). A graph is {\it $k$-regular} if $\delta(G) = \Delta(G) = k$, and {\it regular} if it is $k$-regular for some $k$.
A {\it matching} $M$ in a graph $G$ is a set of pairwise disjoint edges.
The {\it chromatic index} of a graph $G$ is denoted by $\chi'(G)$ and defined as the smallest number of matchings of $G$
the union of which is $E(G)$. Every graph $G$ has $\chi'(G)\in \{\Delta(G),\Delta(G)+1\}$ and graphs with $\chi'(G) = \Delta(G)$ are said to be of Class~$1$.

\begin{lemma}\label{lem:complements-of-line-graphs}
Let $k\ge 2$, let $H$ be a triangle-free $k$-regular graph, and let $G = \overline{L(H)}$.
Then, $G$ is localizable, if and only if $G$ is $k$-localizable, if and only if $\chi'(H) = k$.
\end{lemma}

\begin{proof}
The maximal cliques of the line graph of an arbitrary graph $F$ are exactly the inclusion-wise maximal elements
in ${\cal C} = {\cal T}\cup {\cal S}$, where ${\cal T}$ is the set of (all edge sets of) triangles of $F$ and
${\cal S}$ is the set of all {\it stars} of $F$, that is, sets of edges of the form $E(v) = \{e\in E(F): v$ is an endpoint of $e\}$ for $v\in V(F)$.
Since $H$ is triangle-free, it follows that the maximal cliques of $L(H)$ are exactly the inclusion-wise maximal stars of $H$.
Moreover, since $H$ has no vertices of degree $0$ or $1$, no two stars of $H$ are comparable with respect to
set inclusion, which implies that
the maximal cliques of $L(H)$ are exactly the stars of $H$. Since $H$ is $k$-regular, all
maximal cliques of $L(H)=\overline{G}$ are of size $k$ and consequently $G$ is a well-covered graph with $\alpha(G) = k$.
Therefore, by Proposition~\ref{prop:alpha-chi-bar} $G$ is localizable,
if and only $G$ is $k$-localizable, if and only if
  $\theta(G)= k$.
The statement now follows from the fact that
$\theta (G) = \chi (\overline G) = \chi(L(H)) = \chi'(H)$ is the chromatic index of $H$.
\end{proof}

\begin{theorem}\label{thm:complements-of-line-of-triangle-free}
For every $k\ge 3$, testing whether a given graph that is the complement of the line graph of a triangle-free $k$-regular graph is $k$-localizable (resp.,~localizable) is NP-complete.
\end{theorem}

\begin{proof}
Cai and Ellis showed in~\cite{journals/dam/CaiE91} that for every $k\ge 3$, it is NP-complete to determine whether a given $k$-regular triangle-free graph
is $k$-edge-colorable. Thus, if $k\ge 3$ and $H$ is a given $k$-regular triangle-free graph, Lemma~\ref{lem:complements-of-line-graphs}
implies that $H$ is $k$-edge-colorable if and only if the graph $G = \overline{L(H)}$ is $k$-localizable,
if and only if $G$ is localizable. The claimed NP-hardness follows. The problem is also in NP since
a polynomial certificate of the fact that $G$ is localizable (resp., $k$-localizable) is given by a graph $H$ such that
$G = \overline{L(H)}$ together with a proper $k$-edge coloring of~$H$.
\end{proof}

We remark that Theorem~\ref{thm:complements-of-line-of-triangle-free} also implies Corollary~\ref{cor:alpha}, which was derived from
Theorem~\ref{prop:k-localizable}. However, we keep Theorem~\ref{prop:k-localizable} and its proof since
in Section~\ref{sec:counterexample} we will expand on the construction given in the proof of Theorem~\ref{prop:k-localizable}.

To put the result of Theorem~\ref{thm:complements-of-line-of-triangle-free} in perspective, note that testing if a given graph is well-covered can be done in polynomial time within the class of complements of line graphs. This follows from the
fact that a graph $H$ such that $G = \overline{L(H)}$ can be computed in polynomial time~\cite{MR0424435} and the following easy observation: for a given connected graph $H$ on at least three vertices, $\overline{L(H)}$ is well-covered if and only if there exists a positive integer $k$ such that
for every vertex $v\in V(H)$, we have $d_H(v)\in \{1,k\}$, with $k = 3$ if $H$ contains a triangle.

\subsection{Proof of Theorem~\ref{thm:hard}}

\begin{theoremHardness}[restated]
The problem of recognizing localizable graphs is:
\begin{itemize}
  \item co-NP-complete in the class of weakly chordal graphs (and consequently in the class of perfect graphs),
  \item NP-complete in the class of complements of line graphs of triangle-free $k$-regular graphs, for every fixed $k\ge 3$,
  \item NP-complete in the class of graphs of independence number $k$, for every fixed $k\ge 3$.
\end{itemize}
\end{theoremHardness}

\begin{proof}
The theorem is a direct consequence of Theorem~\ref{thm:weakly-chordal}, Corollary~\ref{cor:alpha} and Theorem~\ref{thm:complements-of-line-of-triangle-free}.
\end{proof}

\begin{sloppypar}
\section{Characterizations}\label{sec:poly}
\end{sloppypar}

In this section, we characterize localizable graphs within the classes of triangle-free graphs, $C_4$-free graphs, cubic graphs, and line graphs.
Our characterizations also imply polynomial time recognition algorithms of localizable graphs within each of these classes.
First, we briefly summarize known results from the literature leading immediately to graph classes in which localizability can be tested in polynomial time. For background on graph classes, we refer to~\cite{MR1686154,MR2063679}.

By Proposition~\ref{prop:alpha-chi-bar}, a graph is localizable if and only if its independent domination and clique cover numbers coincide. Therefore, the class of localizable graphs can be recognized in polynomial time in any class of graphs for which these two parameters are polynomially computable. Examples of such graph classes include the class of circular-arc graphs~\cite{MR1622646,MR1143909} and any class of perfect graphs for which the independent domination problem is polynomially solvable~\cite{MR936633}, for instance chordal graphs~\cite{MR687354}, cocomparability graphs~\cite{MR1229694}, and distance-hereditary graphs~\cite{MR1651039}. More generally, by Corollary~\ref{cor:perfect}, the class of localizable graphs can be recognized in polynomial time in any class of perfect graphs for which well-covered graphs can be recognized in polynomial time, for instance for perfect graphs of bounded degree~\cite{MR1640952}, or for claw-free perfect graphs~\cite{MR1438624,MR1376052}.

Chordal well-covered (equivalently: chordal localizable) graphs were characterized by Prisner et al.~\cite{MR1368737}, as follows. A clique $C$ in a graph $G$ is said to be {\it simplicial} if there exists a vertex $v\in V(G)$ such that $C = N[v]$, the closed neighborhood of $v$. Prisner et al.~showed that a chordal graph is well-covered if and only if each vertex is in a unique simplicial clique. In the same paper~\cite{MR1368737}, Prisner et al.~proved that the same condition characterizes well-covered graphs among {\it simplicial} ones, that is, among graphs in which every vertex is in a simplicial clique: a simplicial graph is well-covered if and only if each vertex is in a unique simplicial clique. Since each simplicial clique is strong, this property implies localizability, and therefore localizable simplicial graphs can also be recognized in polynomial time.

In~\cite[Theorem 4.2]{MR1264476}, Dean and Zito showed that a $C_4$-free semi-perfect graph $G$ is well-covered (equivalently: localizable) if and only if every minimum clique cover ${\cal C}$ of $G$ is a partition of the vertex set and every clique of ${\cal C}$ contains a simplicial vertex. It is not difficult to see that this condition is equivalent to the condition that each vertex is in a unique simplicial clique. Therefore, since every chordal graph is $C_4$-free and semi-perfect, the result of Dean and Zito generalizes the above-mentioned characterization of well-covered (equivalently: localizable) chordal graphs due to Prisner et al.

It is also worth mentioning that well-coveredness (equivalently: localizability) of cocomparability graphs is equivalent to a known (and polynomially verifiable) property of a derived partially ordered set. For a graph $G$, let us denote by ${\cal P}_G$ the set of all partial orders (posets) with ground set $V(G)$ in which two distinct elements are incomparable if and only if they are adjacent in $G$. A graph $G$ is cocomparability if and only if ${\cal P}_G\neq\emptyset$.
A cocomparability graph $G$ is well-covered if and only if some poset in ${\cal P}_G$ is {\it graded} (that is, all its maximal chains are of the same size), if and only if all posets in ${\cal P}_G$ are graded.


Next, it follows from results of~\cite{MR2323400,MR1739644,MR2232389} that both the clique cover number and the independent domination number are polynomially computable for graphs of bounded clique-width. Therefore, by Proposition~\ref{prop:alpha-chi-bar} the same conclusion holds for the problem of recognizing if a given graph of small clique-width is localizable. This generalizes the result for distance-hereditary graphs (which are of clique-width at most $3$~\cite{MR1792124}).

We summarize the above observations in the following theorem.

\begin{theorem}
The problem of recognizing localizable graphs is polynomially solvable within each of the following graph classes:
$C_4$-free semi-perfect graphs, cocomparability graphs, perfect graphs of bounded degree, claw-free perfect graphs,
simplicial graphs, and graphs of bounded clique-width.
\end{theorem}

The results of this section allow to extend the above list by adding to it the classes of triangle-free graphs, $C_4$-free graphs, cubic graphs, and line graphs.

\subsection{Triangle-free graphs and $C_4$-free graphs}\label{sec:triangle-free}

Let $G$ be a triangle-free graph without isolated vertices. Since every strong clique is maximal and $G$ is triangle-free, $G$ is localizable if and only if $G$ has a partition of its vertex set into strong cliques of size two. Theorem~\ref{thm:Favaron} thus immediately implies the following.

\begin{theorem}\label{thm:triangle-free}
For a triangle-free graph $G$ without isolated vertices, the following properties are equivalent:
\begin{enumerate}
  \item $G$ is localizable.
  \item $G$ is very well-covered.
  \item There exists a perfect matching in $G$ every edge of which is a strong clique.
  \item There exists a perfect matching in $G$, and every edge of every perfect matching in $G$ is a strong clique.
\end{enumerate}
\end{theorem}

Benedetti and Varbaro~\cite{MR2821708} studied a property of graphs referred to as the ``matching square condition'', which, for triangle-free graphs without isolated vertices, is equivalent to condition 3 in the Theorem~\ref{thm:triangle-free}. Moreover, Theorem~\ref{thm:triangle-free} and Corollary~\ref{cor:perfect} imply the characterization of well-covered graphs within the class of triangle-free semi-perfect graphs by Dean and Zito~\cite[Theorem 4.3]{MR1264476}. From the algorithmic point of view, the equivalence between the first and the last property in the above list implies that there is a polynomial time algorithm to test if a given triangle-free graph is localizable: After the removal of isolated vertices, one can use Edmonds' algorithm~\cite{MR0177907} to test if the graph has a perfect matching, and if a perfect matching is found, each of its edges is tested for being a strong clique using Lemma~\ref{lem:strong}.

We now turn to a characterization of $C_4$-free localizable graphs. First, we characterize strong cliques in a $C_4$-free graph.

\begin{lemma}\label{lem:C4-free}
A clique in a $C_4$-free graph is strong if and only if it is simplicial.
\end{lemma}

\begin{proof}
It is easy to see that in any graph, every simplicial clique is strong.

For the converse direction, let $C$ be a strong clique in a $C_4$-free graph $G$. Suppose for a contradiction that $C$ is not simplicial.
Then, every vertex of $C$ has a neighbor in $V(G)\setminus C$, that is, $N(C)$ dominates $C$.
Let $I$ be any minimal set of vertices in $N(C)$ that dominates $C$.
We claim that $I$ is an independent set in $G$. Suppose that there exists a pair $x,y$ of adjacent vertices in $I$.
The $C_4$-freeness of $G$ implies that the neighborhoods of $x$ and $y$ in $C$ are comparable, that is, $N(x)\cap C \subseteq N(y)\cap C$ or
$N(y)\cap C \subseteq N(x)\cap C$.
But then, assuming (w.l.o.g.) $N(x)\cap C \subseteq N(y)\cap C$, we could remove $y$ from $I$ to obtain a subset of $N(C)$ that dominates $C$ properly contained in $I$, contradicting the minimality of $I$. This shows that $I$ is independent, as claimed.
Extending $I$ to an arbitrary maximal independent set of $G$ yields a maximal independent set disjoint from $C$, contradicting the fact that $C$ is a strong clique. The obtained contradiction completes the proof that $C$ is simplicial.
\end{proof}

\begin{theorem}\label{thm:C4-free}
A $C_4$-free graph $G$ is localizable if and only if each vertex of $G$ is in a unique simplicial clique.
\end{theorem}

\begin{proof}
By Lemma~\ref{lem:C4-free}, $G$ is localizable if and only if
its vertex set can be partitioned into simplicial cliques.
This condition is clearly satisfied if each vertex is in a unique simplicial clique.

Suppose now that $V(G)$ partitions into simplicial cliques $C_1,\ldots, C_k$.
It suffices to show that for every $v\in C_i$, clique
$C_i$ is the only simplicial clique containing $v$.
This is clear if $v$ is a simplicial vertex (in this case $C_i$ is the only maximal clique containing $v$).
Suppose now that $v$ is not a simplicial vertex and that $C'$ is a simplicial clique containing $v$ such that $C'\neq C_i$.
Let $v'\in C'$ be a vertex such that $N[v'] = C'$.  Consider the simplicial clique $C_j$ such that $v'\in C_j$. We have $j\neq i$, since otherwise
we would have $C_i\subseteq N[v'] = C'$, contrary to the fact that $C_i$ and $C'$ are distinct maximal cliques.
Since $C_j$ is a maximal clique, vertex $v$ has a non-neighbor in $C_j$, say $v''$. But now, vertices
$v$ and $v''$ form a pair of non-adjacent neighbors of $v'$, contrary to the fact that the neighborhood of $v'$ is a clique.
This shows that $C_i$ is the only simplicial clique containing $v$ and completes the proof.
\end{proof}

Since the set of simplicial cliques in a graph can be computed in polynomial time, Theorem~\ref{thm:C4-free} implies the existence of a polynomial time algorithm to determine if a given $C_4$-free graph is localizable. Moreover, since within the class of semi-perfect graphs, localizable graphs coincide with well-covered ones, Theorem~\ref{thm:C4-free} generalizes the above-mentioned characterizations of well-covered graphs within the classes of chordal and of $C_4$-free semi-perfect graphs due to Prisner et al.~\cite{MR1368737} and Dean and Zito~\cite{MR1264476}, respectively.

\subsection{Cubic graphs}\label{sec:cubic}

A graph is {\it cubic} if it is $3$-regular. Well-covered cubic graphs were classified by Campbell et al.~\cite{MR1220613}.
The classification consists of three infinite families together with seven exceptional graphs.
By testing each of the graphs in the list for localizability, a classification of localizable cubic graphs could be derived.
However, the way to characterizing the well-covered cubic graphs was long, building on earlier results characterizing well-covered cubic planar graphs due to Campbell~\cite{Campbell} and Campbell and Plummer~\cite{MR942505} (see Plummer's survey~\cite{MR1254158} for more details).
We give a short direct proof of the classification of localizable cubic graphs.

We first develop a property of strong cliques in regular graphs.

\begin{lemma}\label{prop:trianglefree-regular}
Let $r\ge 1$ and let $G$ be a connected $r$-regular graph. Then $G$ has a strong clique of size two if and only if $G\cong K_{r,r}$.
\end{lemma}

\begin{proof}
Suppose that $G$ is a connected $r$-regular graph, and let $\{u,v\}$ be a strong clique in $G$.
Let $A$ denote the set of neighbours of $u$ different from $v$, and let $B$ denote the set of neighbours of $v$ different from  $u$.
By Lemma~\ref{lem:strong}, $A\cap B=\emptyset$ and every vertex in $A$ is adjacent to every vertex in $B$.
Moreover, since $G$ is $r$-regular, we have $|A|=|B|=r-1$. The connectedness of $G$ and the fact that $G$ is $r$-regular implies that $G\cong K_{r,r}$.
The converse direction is immediate.
\end{proof}

\begin{figure}[!ht]
\centering
\begin{minipage}{.35\textwidth}
  \centering
  \includegraphics[width=0.6\linewidth]{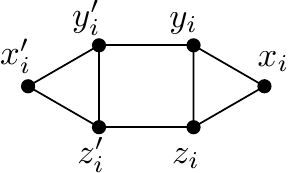}
  \captionof{figure}{The graph $F$}
  \label{fig:F}
\end{minipage}%
\hspace{1cm}
\begin{minipage}{.4\textwidth}
  \centering
  \includegraphics[width=\linewidth]{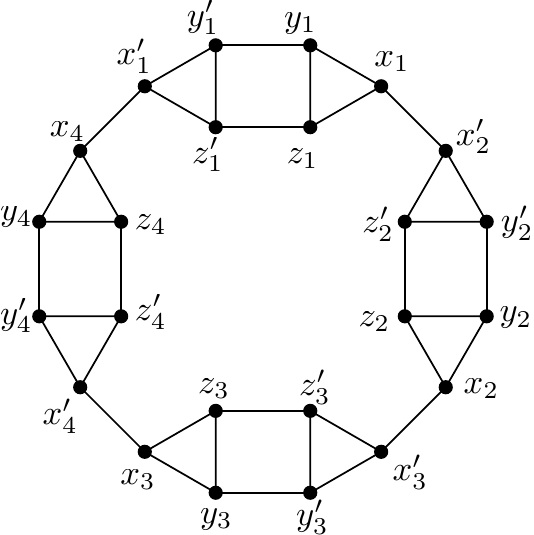}
  \captionof{figure}{The graph $F_4$}
  \label{fig:F4}
\end{minipage}
\end{figure}

For an integer $n\ge 2$, let $F_n$ denote the graph obtained as follows:
take a cycle $v_1v_2\ldots v_n v_1$ of length $n$ (if $n=2$, then $v_1v_2v_1$  is a cycle of length $2$ with $2$ parallel edges),
replace every vertex $v_i$ of the cycle by vertices $x_i,x'_i,y_i,y'_i,z_i,z'_i$ inducing the graph $F$ (see Figure~\ref{fig:F}); replace each edge $v_iv_{i+1}$ ($i=1,\ldots,n-1$) of the cycle by an edge $x_i x'_{i+1}$, finally replace the edge $v_nv_i$ by the edge $x_nx'_1$ (see Figure~\ref{fig:F4} for an example).

\begin{theorem}\label{thm:3-regular}
Let $G$ be a connected cubic  graph. Then, the following statements are equivalent:
\begin{enumerate}
  \item $G$ is localizable.
  \item Every vertex of $G$ is contained in a strong clique.
  \item $G$ is isomorphic to one of the graphs in the set $\{K_{3,3},K_4,\overline{C_6}\}\cup\{F_n: n\ge 2\}$ (see Figure~\ref{fig:cubic}).
\end{enumerate}
In particular, localizable cubic graphs can be recognized in polynomial time.
\end{theorem}

\begin{figure}[!ht]
  \centering
  \includegraphics[width=\linewidth]{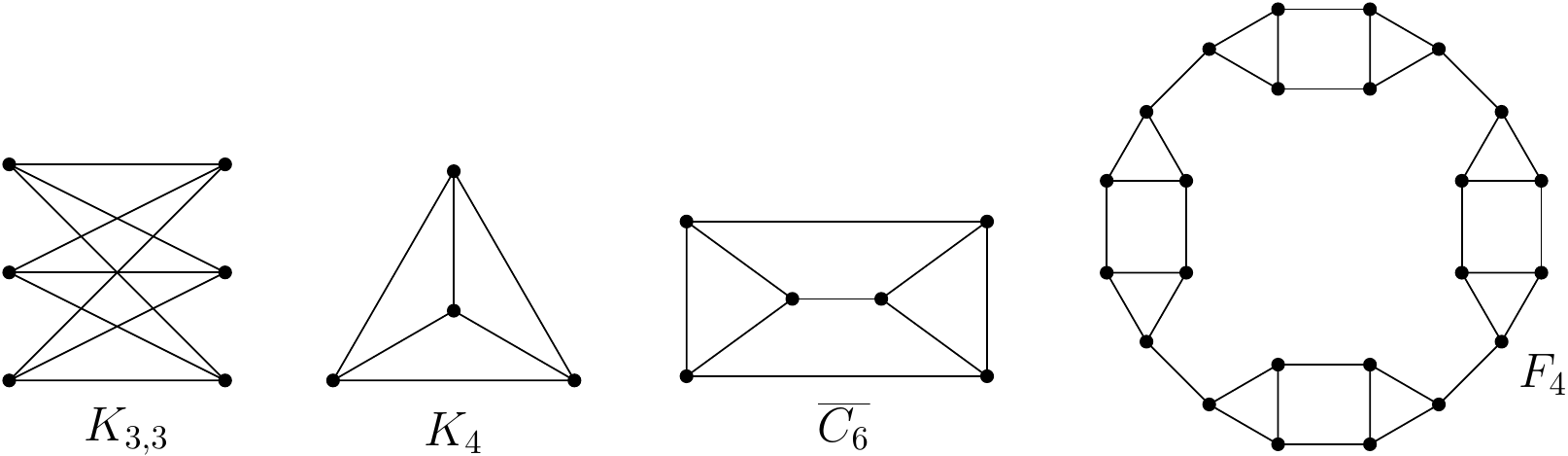}
  \captionof{figure}{Cubic localizable graphs: $K_{3,3}$, $K_4$, $\overline{C_6}$, and an infinite family $\{F_n: n\ge 2\}$}
  \label{fig:cubic}
\end{figure}

\begin{proof}
It follows immediately from the definition of localizable graphs that every vertex of a localizable graphs is contained in a strong clique.
Therefore, statement $1$ implies statement $2$.

It is also easy to check that each of the graphs in the set $\{K_{3,3},K_4,\overline{C_6}\}\cup\{F_n: n\ge 2\}$ is localizable.
(Each edge of $K_{3,3}$ is a strong clique; in $K_4$ the whole vertex set is a strong clique; in each of the remaining graphs
each vertex is contained in a unique triangle and all the triangles are strong.)
Therefore, statement $3$ implies statement $1$.

Finally, we show that statement $2$ implies statement $3$.

Since $G$ is cubic, it cannot have a clique of size $5$ or a strong clique of size $1$.
If there exists a strong clique of size $2$ in $G$, then by Lemma~\ref{prop:trianglefree-regular}, $G\cong K_{3,3}$.
If there exists a strong clique of size $4$ in $G$, then $G\cong K_4$ since $G$ is cubic.

Suppose now that every vertex of $G$ belongs to a strong clique of size $3$. Let $C_1=\{x_1,y_1,z_1\}$ be a strong clique in $G$. All three vertices of $C_1$ cannot have a common neighbour, otherwise $C_1$ would not be a strong clique. Suppose that two vertices of $C_1$, say $x_1$ and $y_1$, have a common neighbour, say $w$, outside $C_1$. Let $w'\ne z_1$ be the remaining neighbour of $w$. Since the only clique of size $3$ that contains $w$ is $\{x_1,y_1,w\}$, it follows that $\{x_1,y_1,w\}$ is a strong clique. If $w'$ is not adjacent to $z_1$, then a maximal independent set in $G$ containing $\{w',z_1\}$ would be disjoint from $\{x_1,y_1,w\}$, contradicting the fact that $\{x_1,y_1,w\}$ is a strong clique. This implies that $w'$ is adjacent to $z_1$. However, in this case, vertex $w'$ does not belong to any triangle in $G$, which contradicts the assumption that every vertex of $G$ belongs to a strong clique of size $3$. This shows that no two vertices of $C_1$ have a common neighbour outside of $C_1$.

If the neighbours of $x_1,y_1,z_1$ outside of $C_1$ form an independent set $S$, then any maximal independent set in $G$ containing $S$ would be disjoint from $C_1$, a contradiction. Suppose now that $y'_1$ and $z'_1$ are the remaining neighbours of $y_1$ and $z_1$, and that $y'_1$ is adjacent to $z'_1$.
Since every vertex belongs to a triangle, it follows that $y'_1$ and $z'_1$ have a common neighbour, say $x'_1$. If $x_1$ is adjacent to $x'_1$ then $G\cong \overline{C_6}$. If $x_1$ and $x'_1$ are not adjacent, then it follows from the above that for any strong clique $C=\{u,v,w\}$ in $G$, there exist vertices $u',v',w'$ such that $vv', ww'\in E(G)$, $\{u',v',w'\}$ is a strong clique and $u$ is not adjacent to $u'$.

Let $x'_2$ be the remaining neighbour of $x_1$. Since $x'_2$ belongs to a strong clique, it follows that there exist vertices $y'_2$ and $z'_2$ such that $C_2=\{x'_2,y'_2,z'_2\}$ is strong clique. Note that the fact that $G$ is cubic implies that $\{y'_2,z'_2\}\cap \{x_1,y_1,z_1,x_1',y_1',z_1'\} = \emptyset$.
Let $y_2$ and $z_2$ be the remaining neighbours of $y'_2$ and $z'_2$ respectively.
Then using the same argument as above, replacing $C$ with $C_2$ one can easily see that $y_2$ and $z_2$ are adjacent, and that they have a common neighbour, say $x_2$. If $x_2$ is adjacent with $x'_1$, then $G\cong F_2$. If $x_2$ and $x'_1$ are not adjacent, then let $x'_3$ be the remaining neighbour of $x_2$ and repeat the same argument as before. Since the graph is finite and connected, it follows that for some $n$, we will have that $x_n$ is adjacent with $x_1'$, hence $G\cong F_n$.
\end{proof}

\subsection{Line graphs}\label{sec:line}

A graph $G$ is said to be a {\it line graph} if $G = L(H)$ for some graph $H$. In this section we characterize localizable line graphs.
The characterization implies the existence of a polynomial time algorithm to determine whether a given line graph is localizable.

Recall that the line graph of a graph $H$ is well-covered if and only if $H$ is {\it equimatchable}, that is, if all maximal matchings of $G$ are of the same size. The question of characterizing equimatchable graphs was posed by Gr\"{u}nbaum in $1974$~\cite{Gruenbaum}; in the same year, equimatchable graphs were studied and characterized by Lewin~\cite{Lewin} and by Meng~\cite{Meng}, and shown to be polynomially recognizable by Lesk et al.~\cite{MR777180}.
Since every localizable graph is well-covered, every graph the line graph of which is localizable is equimatchable.
Moreover, since the line graphs of bipartite graphs are perfect, Corollary~\ref{cor:perfect} implies that
the line graph of a bipartite graph $H$ is localizable if and only if it is well-covered; equivalently, if $H$ is equimatchable.
Lesk et al.~\cite{MR777180} characterized equimatchable bipartite graphs as follows.

\begin{theorem}[\!\!\cite{MR777180}]\label{thm:bip-equim}
A connected bipartite graph $H$ with a bipartition of its vertex set into two independent sets $V(H) = U\cup W$
with $|U|\le |W|$ is equimatchable if and only if for all $u\in U$, there exists a non-empty set $X\subseteq N(u)$ such that
$|N(X)|\le |X|$.
\end{theorem}

Given a graph $H$, we say that a vertex $v\in V(H)$ is {\it strong} if every maximal matching of $H$ covers $v$.
The above characterization has the following consequence.

\begin{lemma}\label{lem:bip-equim}
A connected bipartite graph $H$ is equimatchable if and only if it has a bipartition of its vertex set into two independent sets $V(H) = U\cup W$
such that all vertices of $U$ are strong.
\end{lemma}

\begin{proof}
First, we show that a vertex $u\in V(H)$ in a bipartite graph $H$ is strong if and only if
there exists a non-empty set $X\subseteq N(u)$ such that $|N(X)|\le |X|$.
This is equivalent to showing that $u$ is not strong if and only if
$|N(X)| > |X|$ for all non-empty sets $X\subseteq N(u)$, which is further equivalent to
the condition that $|N_{H-u}(X)| \ge |X|$ for all non-empty sets $X\subseteq N_H(u)$.
By Hall's Theorem~\cite{Hall01011935}, this is equivalent to the existence of a
matching $M$ in the graph $H-u$ such that every vertex of $N_H(u)$ is incident with an edge in $M$.
This is in turn equivalent to the condition that
$H$ contains a maximal matching not covering $u$, that is, that
$u$ is not strong in $H$, as claimed.

Let $H$ be a connected bipartite graph. Suppose first that $H$ is equimatchable and fix a bipartition of its vertex set into two independent sets $V(H) = U\cup W$ with $|U|\le |W|$. Theorem~\ref{thm:bip-equim} implies that for all $u\in U$, there exists a non-empty set $X\subseteq N(u)$ such that
$|N(X)|\le |X|$. By the above equivalence, all vertices in $U$ are strong. Conversely, suppose that $H$ is a connected bipartite graph with a bipartition of its vertex set into two independent sets $V(H) = U\cup W$ such that all vertices of $U$ are strong. By the above equivalence, for all $u\in U$, there exists a non-empty set $X\subseteq N(u)$ such that $|N(X)|\le |X|$. Let $M$ be a maximal matching of $H$. By assumption on $U$, every vertex of $U$ is incident with an edge of $M$. This implies that $|U| = |M|\le |W|$. Thus, by Theorem~\ref{thm:bip-equim}, $H$ is equimatchable.
\end{proof}

As pointed out above, the line graph of every equimatchable bipartite graph is localizable.
In what follows, we will show that graphs the line graphs of which is localizable do not differ too much
from equimatchable bipartite graphs. To describe the result, we need a couple of definitions.

\begin{definition}\label{def:pendant}
The \emph{diamond} is the graph with vertices $a,b,c,d$ and edges $ab,ac,bc,bd,cd$; the vertices $a$ and $d$ are its \emph{tips}.
Let $T$ be a triangle in a graph $H$ and let $a,b,c$ be the vertices of $T$.
We say that $T$ is:
\begin{itemize}
  \item a \emph{pendant triangle} of $H$ if $d_H(a) = d_H(b) = 2<d_H(c)$,
  \item \emph{contained in a pendant diamond} if $H$ has a subgraph
with vertices $a,b,c,d$ inducing a diamond with tips $a$ and $d$ such that $d_H(a) = 2$ and $d_H(b) = d_H(c) = 3\le d_H(d)$, and
  \item \emph{contained in a pendant $K_4$} if $H$ has a subgraph with vertices $a,b,c,d$ inducing a $K_4$ such that $d_H(a) = d_H(b) = d_H(c) = 3<d_H(d)$.
\end{itemize}
Any diamonds and $K_4$s as above will be referred to as \emph{pendant diamonds} and \emph{pendant $K_4$s} (of $H$), respectively.
Pendant triangles, pendant diamonds, and pendant $K_4$s of $H$ will be referred to briefly as its \emph{pendant subgraphs}; see Figure~\ref{fig:triangles}. Note that if a graph $H$ has a pendant subgraph, then $H$ is not isomorphic to any graph in the set $\{K_3$, $K_4$, diamond$\}$. Moreover,
every pendant subgraph of $H$ has a unique {\it root}, that is, a vertex connecting it to the rest of the graph.
We denote by $\Rt(H)$, $\Rd(H)$, and $\RK(H)$ the sets of roots of all pendant triangles, pendant diamonds, and
pendant $K_4$s of $H$, respectively. The \emph{pendant reduction} of $H$ is the graph obtained from $H$ by deleting
from $H$ all non-root vertices of its pendant subgraphs.
\end{definition}

\begin{figure}[!ht]
  \begin{center}
\includegraphics[width=0.75\linewidth]{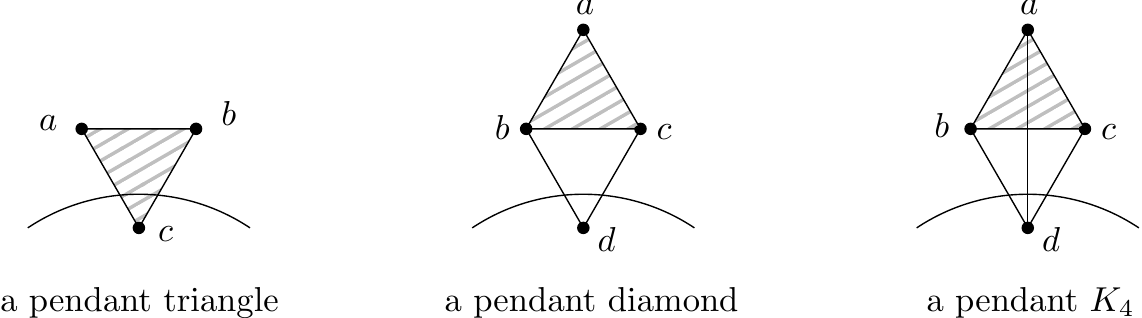}
  \end{center}
\caption{The three types of pendant subgraphs. The shaded triangles represent a pendant triangle, a triangle contained in a pendant diamond, and
a triangle contained in a pendant $K_4$, respectively.} \label{fig:triangles}
\end{figure}

The characterization of localizable line graphs is given by the following theorem.

\begin{theorem}\label{thm:line}
Let $H$ be a connected graph and let $G = L(H)$. Then, $G$ is localizable if and only if one of the following holds:
\begin{enumerate}
\item $H$ is not isomorphic to any graph in the set $\{K_3$, $K_4$, diamond$\}$ and its pendant reduction
$F$ is a connected bipartite graph with a bipartition of its vertex set into two independent sets $U$ and $W$ such that
  \begin{enumerate}[(a)]
    \item $\Rd(H)\cup \RK(H)\subseteq U$,
     \item $\Rt(H) \subseteq W$, and
 \item each vertex $u\in U\setminus \RK(H)$ is strong in the graph $F-(N_F(u)\cap \Rt(H))$.
  \end{enumerate}
    \item $H$ is isomorphic to either $K_3$ or $K_4$.
\end{enumerate}
\end{theorem}

Before proving Theorem~\ref{thm:line}, let us note two of its consequences.

\begin{corollary}\label{cor:line-of-triangle-free}
Let $H$ be a connected triangle-free graph and let $G = L(H)$.
Then, $G$ is localizable if and only if $H$ is an equimatchable bipartite graph.
\end{corollary}

\begin{proof}
If $H$ is triangle-free, then its pendant reduction equals $H$ and
\hbox{$\Rd(H)= \RK(H)= \Rt(H)  = \emptyset$}. Thus, the statement of the corollary is an immediate consequence of
Theorem~\ref{thm:line} and Lemma~\ref{lem:bip-equim}.
\end{proof}

\begin{corollary}\label{cor:line}
There exists a polynomial time algorithm for the problem of determining if a given line graph is localizable.
\end{corollary}

\begin{proof}
Let $G$ be a given line graph. Since a graph $H$ such that $G = L(H)$ can be computed in linear time~\cite{MR0424435}, we may assume that we know $H$.
Moreover, since $G$ is localizable if and only if each component of $G$ is localizable, we may assume that $G$ (and therefore $H$) is connected.
We may also assume that $H$ is not isomorphic to any graph in the set $\{K_3$, $K_4$, diamond$\}$.

We apply Theorem~\ref{thm:line}. We compute the pendant subgraphs of $H$, the sets $\Rt(H)$, $\Rd(H)$, and $\RK(H)$ of their roots, and the pendant reduction $F$ of $H$. We may also assume that $F$ is a connected bipartite graph, with parts $U$ and $W$. Under these assumptions,
to verify whether $L(H)$ is localizable we only need to verify whether there exists a part of the bipartition of $F$, say $U$, satisfying the following conditions:
\begin{enumerate}[(a)]
  \item $\Rd(H)\cup \RK(H)\subseteq U$.
  \item $\Rt(H) \cap U = \emptyset$, and
  \item each vertex $u\in U\setminus \RK(H)$ is strong in the graph $F-(N_F(u)\cap \Rt(H))$.
\end{enumerate}
All the above computations as well as the verification of conditions (a) and (b) can be carried out in linear time.
It only remains to justify that conditions (c) can be tested in polynomial time. Note that
a vertex $u\in U\setminus \RK(H)$ is not strong in the graph $F' = F-(N_F(u)\cap \Rt(H))$ if and only if
$F'$ has a matching $M$ such that every neighbor of $u$ is incident with an edge of $M$.
The existence of such a matching can be determined by a maximum matching computation in the
(bipartite) subgraph of $F'$ induced by the vertices of distance $1$ or $2$ from $u$.
Thus, condition (c) can be verified in polynomial time by carrying out $O(|U|)$ bipartite matching computations,
which is well-known to be polynomially solvable (see, e.g.,~\cite{MR0337699}).
\end{proof}

In the rest of this section, we prove Theorem~\ref{thm:line}. This will be done through a sequence of lemmas.
Our first lemma gives a translation of the property that $G = L(H)$ is localizable to the graph $H$. Recall that a vertex $v\in V(H)$ is said to be {\it strong} if every maximal matching of $H$ covers $v$. We say that a triangle $T$ in a graph $H$ is {\it strong} if every maximal matching of $H$ contains an edge of $T$. A \emph{decomposition} of a graph $H$ is a set ${\cal F}$ of subgraphs of $H$ such that each edge of $H$ appears in exactly one subgraph from ${\cal F}$ (we also say that ${\cal F}$ \emph{decomposes} $H$).

\begin{lemma}\label{lem:line}
Let $G = L(H)$. Then, $G$ is localizable if and only if $H$ contains an independent set $S$ of strong vertices
and a set ${\cal T}$ of strong triangles decomposing $H-S$.
\end{lemma}

\begin{proof}
Consider a graph $H$ and suppose that its line graph, $G= L(H)$, is localizable. Thus, the vertices of $G$ can be partitioned into $k=\alpha(G)$ strong cliques $C_1,\ldots,C_k$. Each clique in $G$ corresponds to either a triangle in $H$ or to a star of $H$, that is, to a set of edges of the form $E(v) = \{e\in E(H): v$ is an endpoint of $e\}$ for some $v\in V(H)$ (in which case we say that $E(v)$ is the star {\it centered at $v$}). Moreover, these triangles $T_1,\ldots,T_{k'}$ and stars $S_{k'+1},\ldots,S_k$ form a partition of the edge set of $H$. It follows that the centers of the stars form an independent set, say $S$, of $H$, and the triangles $T_1,\ldots,T_{k'}$ are subgraphs of $H-S$ decomposing $H-S$.
Since every maximal independent set of $G$ intersects every strong clique $C_1,\ldots,C_k$, we deduce that every maximal matching in $H$ intersects every triangle $T_i$ and every star $S_j$ for $i=1,\ldots,k'$ and $j=k'+1,\ldots,k$, that is, each triangle $T_i$ is strong and each vertex in $S$ is strong.

Conversely, suppose that $H$ contains an independent set $S$ of strong vertices and a set ${\cal T}$ of strong triangles decomposing $H-S$. Every triangle in $\mathcal{T}$ corresponds to a clique in $G=L(H)$ and every star with center $s\in S$ also corresponds to a clique in $G$. This clearly gives us a partition of the vertex set of $G$ into cliques $C_1,\ldots,C_k$. Furthermore, since every maximal matching in $H$ covers $S$ and contains an edge from each triangle in ${\cal T}$, it follows that every maximal independent set of $G$ intersects every maximal clique $C_1,\ldots,C_k$, and thus each of these cliques is strong. We conclude that $G$ is localizable.
\end{proof}

Given a graph $H$ such that $L(H)$ is localizable, a pair $(S,{\cal T})$ such that $S$ is an independent set of strong vertices in $H$ and ${\cal T}$ is a set of strong triangles decomposing $H-S$ will be referred to as a {\it line-localizability certifier} of $H$.

The next lemma characterizes strong triangles of $H$. A {\it bull} is a graph with vertices $a,b,c,d,e$ and edges $ab,bc,cd,be,ce$.
A triangle $T$ of $H$ is said to be {\it contained in a bull} if there exists a subgraph of $H$ isomorphic
to a bull that contains $T$ as a subgraph.

\begin{lemma}\label{lem:strong-triangles}
For every triangle $T$ in a graph $H$, the following conditions are equivalent:
 \begin{enumerate}
   \item $T$ is strong in $H$.
   \item $T$ is not contained in any bull.
   \item $T$ is either a pendant triangle, or is contained in a pendant diamond or in a pendant $K_4$.
 \end{enumerate}
\end{lemma}

\begin{proof}
Let $T$ be a triangle in $H$. Clearly, $T$ is not strong if and only it there exists a matching $M$ in $H$ not containing any edge of $T$ and covering at least two vertices of $T$. A minimal matching with such properties consists of two edges, each of which is incident with a vertex of $T$, and hence forms a bull together with $T$.
Therefore, $T$ is not strong if and only if $T$ is contained in a bull.

If $T$ is either a pendant triangle, or is contained in a pendant diamond or in a pendant $K_4$, then $T$ is not contained in any bull. Conversely, suppose that $T$ is not contained in any bull, and let $U$ be the set of vertices in $T$ with a neighbor outside $T$. If $|U| \le 1$ then $T$ is a pendant triangle. If $|U| = 2$ then, since $T$ is not contained in any bull, there is a unique vertex outside $T$ with a neighbor in $T$, and hence $T$ is contained in a pendant diamond. Similarly, if $|U| = 3$ then $T$ is contained in a pendant $K_4$. This establishes the equivalence of the last two conditions.
\end{proof}

The next lemma shows that with the exception of two small cases, the strong triangles in $H$ are pairwise edge-disjoint.

\begin{lemma}\label{lem:K4-diamond}
Let $H$ be a connected graph that is not isomorphic to either $K_4$ or to the diamond.
Then, any two strong triangles in $H$ are edge-disjoint.
\end{lemma}

\begin{proof}
Let $H$ be a connected graph with a pair $T=\{a,b,c\}$ and $T'= \{a,b,d\}$ of strong triangles sharing an edge (namely, $ab$). By Lemma~\ref{lem:strong-triangles}, each of $T$ and $T'$ is either a pendant triangle, or is contained in a pendant diamond or in a pendant $K_4$. Since each of $T$ and $T'$ has at most one vertex of degree $2$, none of them can be a pendant triangle. Suppose first that one of them, say $T$, is contained in a pendant diamond. Then the vertex set of this diamond is exactly $\{a,b,c,d\}$, and the only remaining possibility for $T'$ is that it is also contained in a pendant diamond. Since $H$ is connected, we infer that $H$ is isomorphic to a diamond in this case.
If $T$ is contained in a pendant $K_4$, then, similarly, the vertex set of this $K_4$ is exactly $\{a,b,c,d\}$, and the only remaining possibility for $T'$ is that it is also contained in a pendant $K_4$, hence $H\cong K_4$ in this case.
\end{proof}

Let $H$ be a connected graph that is not isomorphic to any graph in the set $\{K_3$, $K_4$, diamond$\}$.
Recall that the \emph{pendant reduction} of $H$ is the graph obtained from $H$ by deleting
from $H$ all non-root vertices of its pendant subgraphs.
The next lemma establishes some necessary conditions for $L(H)$ to be localizable.

\begin{lemma}\label{lem:line-necessary}
Let $H$ be a connected graph not isomorphic to any graph in the set $\{K_3$, $K_4$, diamond$\}$
such that $L(H)$ is localizable and let $(S,{\cal T})$ be a line-localizability certifier of $H$. Then:
\begin{enumerate}
  \item ${\cal T}$ is the set of all strong triangles of $H$.
  \item $\Rt(H)\cap S = \emptyset$.
  \item $\Rd(H)\cup \RK(H)\subseteq S$.
  \item The pendant reduction $F$ of $H$ is a connected bipartite graph with a bipartition of its vertex set $\{S,V(F)\setminus S\}$.
\end{enumerate}
\end{lemma}

\begin{proof}
Let $H$, $S$, and ${\cal T}$ be as in the statement of the lemma.
Let $\Ptriangle$ denote the set of pendant triangles of $H$, and let $\Pdiamond$ and $\PK$ denote the sets of triangles of $H$ contained in a pendant diamond or in a pendant $K_{4}$, respectively. By Lemma~\ref{lem:strong-triangles}, each strong triangle of $H$ belongs to one (and then to exactly one) of the sets $\Ptriangle$, $\Pdiamond$, and $\PK$. Moreover, by Lemma~\ref{lem:K4-diamond}, any two triangles in $\Ptriangle\cup \Pdiamond\cup \PK$ are pairwise edge-disjoint.

By Lemma~\ref{lem:strong-triangles}, in order to show statement 1, we need to show that ${\cal T} = \Ptriangle\cup \Pdiamond\cup \PK$.
Since every triangle in ${\cal T}$ is strong, we have ${\cal T}\subseteq \Ptriangle\cup \Pdiamond\cup \PK$.
We prove the converse inclusion $\Ptriangle\cup \Pdiamond\cup \PK\subseteq {\cal T}$ in three steps.
Along the way we will also prove statements 2 and 3 above.

First, let $T=\{a,b,c\}\in \Ptriangle$ be a pendant triangle with root $c$. Note that none of the vertices $a$ and $b$ is strong.
Therefore, $\{a,b\}\cap S= \emptyset$ and consequently there is a strong triangle $T'\in {\cal T}$ such that $\{a,b\}\subseteq T'$. Since $T$ is the only triangle of $H$ containing the edge $ab$, we infer that $T' = T$ and hence $T\in {\cal T}$.
This shows that $\Ptriangle\subseteq {\cal T}$. Moreover, since $T\in {\cal T}$, we have $T\cap S = \emptyset$ and consequently $c\not\in S$.
This shows statement 2, that is, that no root of a pendant triangle is in $S$.

Second, let $T=\{a,b,c\}\in \Pdiamond$ be a triangle contained in a pendant diamond with root $d$ such that $d_H(a) = 2$, $d_H(b)=d_H(c)=3$. Since vertex $a$ is not strong, we have $a\not\in S$. Since $S$ is independent, one of $b$ and $c$, say $b$, does not belong to $S$. It follows that the edge $ab$ is contained in some triangle $T'\in {\cal T}$. Again, since $T$ is the only triangle containing the edge $ab$, we infer that $T' = T$ and hence $T\in {\cal T}$. This shows that $\Pdiamond\subseteq {\cal T}$. By Lemma~\ref{lem:K4-diamond}, the strong triangles of $H$ are pairwise edge-disjoint, therefore
the triangle $\{b,c,d\}$ is not strong. It follows that $d\in S$ since otherwise the edge $bd$ would be an edge of $H-S$ not covered by any triangle in ${\cal T}$.

Now, let $T=\{a,b,c\}\in \PK$ be a triangle contained in a pendant $K_4$ with root $d$ and $d_H(a)=d_H(b)=d_H(c)=3$. Suppose for a contradiction that $T\not\in {\cal T}$. Since $S$ is independent, it contains at most one vertex of $T$. We may thus assume that $\{a,b\}\cap S = \emptyset$. Let $T'$ be the triangle in ${\cal T}$ covering the edge $ab$. Since $T\not\in {\cal T}$, we infer that $T' \neq T$ and therefore $T' = \{a,b,d\}$. Since the strong triangles of $H$ are pairwise edge-disjoint and $T$ is a strong triangle sharing the edge $ab$ with $T'$, we infer that $T'$ is not strong, contrary to the fact that $T'\in {\cal T}$.
This shows that $\PK\subseteq {\cal T}$.
Similarly as above, using the fact that the strong triangles of $H$ are pairwise edge-disjoint, we infer that
in order to cover the edge $ad$, we must have $d\in S$.
The above two paragraphs show statement 3, that is, that
the root of every pendant diamond or $K_4$ is in $S$.

Finally, we show statement 4. Let $F$ be the pendant reduction of $H$. Since $H$ is connected and $F$ differs from $H$ only by non-root vertices of its pendant subgraphs, we infer that $F$ is connected. Since ${\cal T}$ is the set of all strong triangles of $H$ and $T\cap S = \emptyset$ for all $T\in {\cal T}$, it follows that $S\subseteq V(F)$; in particular, $S$ is an independent set of $F$. It thus suffices to show that the set $V(F)\setminus S$ is independent. We have $V(F) = Z\cup \Rt(H)$, where $Z$ is the set of vertices of $H$ not contained in any strong triangle.
Suppose for a contradiction that there is an edge $uv\in E(H)$ connecting two vertices of $V(F)\setminus S$.
Then $uv$ is an edge of the graph $H-S$. If $\{u,v\}\cap Z \neq \emptyset$, then $uv$ would be an edge of $H-S$ not contained in any triangle in ${\cal T}$, contrary to the assumption on ${\cal T}$. Thus, we have $\{u,v\}\subseteq \Rt$, therefore $u$ and $v$ are adjacent roots of two pendant triangles. Since $uv$ is an edge of $H-S$, it is contained in some triangle $T\in {\cal T}$, which is clearly impossible due to the characterization of strong triangles given by Lemma~\ref{lem:strong-triangles}. This contradiction proves statement~$4$.
\end{proof}

Now we have everything ready to prove the announced characterization of localizable line graphs, which we restate here for convenience.

\begin{theoremLine}[restated]
Let $H$ be a connected graph and let $G = L(H)$. Then, $G$ is localizable if and only if one the following holds:
\begin{enumerate}
\item $H$ is not isomorphic to any graph in the set $\{K_3$, $K_4$, diamond$\}$ and its pendant reduction
$F$ is a connected bipartite graph with a bipartition of its vertex set into
  two independent sets $U$ and $W$ such that
   \begin{enumerate}[(a)]
    \item $\Rd(H)\cup \RK(H)\subseteq U$,
     \item $\Rt(H) \subseteq W$, and
 \item each vertex $u\in U\setminus \RK(H)$ is strong in the graph $F-(N_F(u)\cap \Rt(H))$.
  \end{enumerate}
 \item $H$ is isomorphic to either $K_3$ or $K_4$.
\end{enumerate}
\end{theoremLine}

\begin{proof}
Suppose first that $G=L(H)$ is localizable, and let $(S,{\cal T})$ be a line-localizability certifier of $H$.
Suppose that $H$ is not isomorphic to any of $K_3$ or $K_4$.
Since the line graph of the diamond is not localizable, $H$ is also not isomorphic to the diamond.
Let $F$ be the pendant reduction of $H$.
By construction, we have $\Rt(H)\cup \Rd(H)\cup \RK(H)\subseteq V(F)$.
By Lemma~\ref{lem:line-necessary}, $F$ is a connected bipartite graph with a bipartition of its vertex set $\{S,V(F)\setminus S\}$.
The same lemma implies that $\Rd(H)\cup \RK(H)\subseteq S$ and $\Rt(H) \subseteq V(F)\setminus S$.
Thus, letting $W' = N_F(u)\cap \Rt(H)$ and $F' = F-W'$, we only need to show that
each vertex $u\in S\setminus \RK(H)$ is strong in graph $F'$, and the desired conclusion will follow by taking
$U = S$ and $W = V(F)\setminus S$.

Suppose for a contradiction that there exists a vertex $u\in S\setminus \RK(H)$ that is not strong
in graph $F'$. Then there exists a maximal matching $M'$ of $F'$ such that $u$ is not incident with any edge of $M'$.
Since $u\in S$, vertex $u$ is strong in $H$, that is, every maximal matching of $H$ covers $u$.
However, we will now show that $M'$ is contained in a maximal matching of $H$ not covering $u$.
Since $u\in S$ and $\Rt(H) \subseteq V(F)\setminus S$, vertex $u$ is not the root of any pendant triangle of $H$. Moreover,
since $u\not\in \RK(H)$, it is also not the root of any pendant $K_4$.
Let ${\cal D}$ be the (possibly empty) set of pendant diamonds in $H$ having $u$ as the root.
For each pendant diamond $D\in {\cal D}$, let $e_D$ be the edge of $D$ completing a triangle with $u$.
For each vertex $w\in W'$, let ${\cal T}_w = \{T_1,\ldots, T_k\}$ be the set of
pendant triangles of $H$ with root $w$, let $T_i = \{w,a_i,b_i\}$, and let
$M_w=\{wa_1\}\cup\{a_ib_i: 2\le i\le k\}$. Note that $M_w$ is a matching in $H$.
Let $$M = M'\cup \{e_D: D\in {\cal D}\}\cup \bigcup_{w\in W'}M_w\,.$$
Then, $M$ is a matching in $H$ covering all neighbors of $u$. Therefore, $M'$ is contained in a maximal matching of $H$ not covering $u$,
contradicting the fact that $u$ is strong in $H$.
This completes the proof of the forward direction.

For the converse direction, suppose that one of conditions 1 and 2 in the statement of the theorem holds.
If $H$ is isomorphic to one of $K_3$ and $K_4$, then its line graph is localizable.
Suppose now that condition 1 holds, that is, $H$ is not isomorphic to any graph in the set $\{K_3$, $K_4$, diamond$\}$ and its pendant reduction
$F$ is a connected bipartite graph with a bipartition of its vertex set into
  two independent sets $U$ and $W$ such that conditions (a)--(c) hold.
Let ${\cal T}$ be the set of all strong triangles of $H$.
We will show that the pair $(U,{\cal T})$ is a line-localizability certifier of $H$, which will imply that $G$ is localizable by
Lemma~\ref{lem:line}.
In other words, we need to show that $U$ is an independent set of strong vertices in $H$ and
that ${\cal T}$ is a set of strong triangles decomposing $H-U$.
By Lemma~\ref{lem:strong-triangles}, ${\cal T}$ equals to the set containing all pendant triangles, all triangles
contained in a pendant diamond, and all triangles contained in a pendant $K_4$.
Note that by Lemma~\ref{lem:K4-diamond}, any two strong triangles in $H$ are edge-disjoint.
Therefore, the fact that $F$ is bipartite and properties (a) and (b)
imply that ${\cal T}$ decomposes $H-U$.
Clearly, $U$ is an independent set in $H$.
It therefore remains to show that every vertex $u\in U$ is strong in $H$.
Suppose for a contradiction that some $u\in U$ is not strong in $H$.
Then, there exists a maximal matching $M$ of $H$ not covering $u$.
We consider two cases.
Suppose first that $u\in U\setminus \RK(H)$.
By property (c), vertex $u$ is strong in the graph $F' = F-(N_F(u)\cap \Rt(H))$.
It follows that matching $M' = M\cap E(F')$ covers all vertices in
$N_F(u)\setminus \Rt(H)$, which are exactly the neighbors of $u$ in $F'$.
Extending $M'$ to a maximal matching of $F'$ not covering $u$ shows that
$u$ is not strong in $F'$, a contradiction.
Suppose now that $u\in \RK(H)$.
Let $D$ be a pendant $K_4$ of $H$ with vertex set $\{a,b,c,u\}$ and with root $u$.
Since $u$ is not covered by $M$, matching $M$ covers at most two of the vertices
in $\{a,b,c\}$. Adding to $M$ the edge connecting $u$ with an uncovered vertex in $\{a,b,c\}$
results in a matching of $H$ properly containing $M$, which contradicts the maximality of $M$.
This completes the proof.
\end{proof}

\section{Counterexamples to a conjecture by Zaare-Nahandi}\label{sec:counterexample}

In~\cite{MR3356635}, Zaare-Nahandi posed the following conjecture.
Recall that a graph $G$ is said to be semi-perfect if $\theta(G)= \alpha(G)$.

\begin{conjecture1}[restated]
Let $G$ be an $s$-partite well-covered graph in which all maximal cliques are of size $s$. Then $G$ is semi-perfect.
\end{conjecture1}

The conjecture can be equivalently stated in terms of localizable graphs as follows.

\begin{conjecture2}[restated]
Let $G$ be a localizable co-well-covered graph.
Then $\overline{G}$ is localizable.
\end{conjecture2}

\begin{proposition}
Conjectures~\ref{conj:Zaare-Nahandi} and~\ref{conj2} are equivalent.
\end{proposition}

\begin{proof}
First we prove that
Conjecture~\ref{conj:Zaare-Nahandi} implies  Conjecture~\ref{conj2}.
Assume the validity of Conjecture~\ref{conj:Zaare-Nahandi}.
Let $G$ be a localizable co-well-covered graph.
By Proposition~\ref{prop:alpha-chi-bar}, this is equivalent to:
$G$ and $\overline G$ are well-covered and
$\theta(G)= \alpha(G)$. Let $s = \alpha(G)$.
Then $\chi(\overline G) = s$, hence $\overline{G}$ is an $s$-partite
well-covered graph in which all maximal cliques are of size $s$.
By Conjecture~\ref{conj:Zaare-Nahandi}, $\overline G$ is semi-perfect, that is,
$\theta(\overline G)= \alpha(\overline G)$.
Therefore, since $\overline{G}$ is well-covered, $\overline{G}$ is localizable by Proposition~\ref{prop:alpha-chi-bar}.
Thus, Conjecture~\ref{conj2} holds.

Conversely,  Conjecture~\ref{conj2} implies Conjecture~\ref{conj:Zaare-Nahandi}.
Assume the validity of Conjecture~\ref{conj2}.
Let $G$ be an $s$-partite well-covered graph in which all maximal cliques are of size $s$.
Then $\chi(G) = \omega(G) = s$ and all the maximal independent sets of $\overline G$ are of size $s$.
These conditions are equivalent to: $\theta(\overline{G})= \alpha(\overline G)$ and
 both $G$ and $\overline G$ are well-covered.
By Proposition~\ref{prop:alpha-chi-bar}, this is equivalent to:
$\overline G$ is a localizable co-well-covered graph.
Conjecture~\ref{conj2} implies that $G$ is localizable, and by Proposition~\ref{prop:alpha-chi-bar},
$G$ is semi-perfect. Thus, Conjecture~\ref{conj:Zaare-Nahandi} holds.
\end{proof}

In what follows, we disprove Conjecture~\ref{conj:Zaare-Nahandi} in two different ways.
First, we  give an infinite family of counterexamples showing that even the following conjecture,
which would be a consequence of Conjectures~\ref{conj:Zaare-Nahandi} and~\ref{conj2}, fails.

\begin{conjecture}\label{conj3}
Let $G$ be a localizable co-well-covered graph. Then $\overline{G}$ has a strong clique.
\end{conjecture}

\begin{theorem}\label{thm:counterexamples}
Conjecture~\ref{conj3} is false. Consequently, Conjectures~\ref{conj:Zaare-Nahandi} and~\ref{conj2} are false.
\end{theorem}

\begin{proof}
Let $C$ be an odd cycle of length at least $5$, and let $G$ be the corona of $C$, that is, the graph obtained from $C$
by adding to it $|V(C)|$ new vertices, each adjacent to a different vertex of cycle (and there are no other edges in $G$).
Each of the $|V(C)|$ sets $N[v]$, where $v\in V(G)\setminus V(C)$, is a simplicial (hence strong) clique in $G$.
These cliques are pairwise disjoint and their union is $V(G)$. Therefore, $G$ is localizable.
Since every maximal clique of $G$ has size $2$, we infer that $G$ is co-well-covered.
We claim that $\overline{G}$ does not have any strong clique, which is equivalent to stating that $G$ does not have any strong independent set.
This is true since any strong independent set in $G$ would have to contain a vertex from each of the edges of $C$
(as these edges are all maximal cliques in $G$); however, as $C$ is not bipartite, no independent set of $C$ hits all edges of $C$.
\end{proof}

Second, we give a `computational complexity' disproof of Conjecture~\ref{conj:Zaare-Nahandi}, by showing
that not only there are localizable co-well-covered graphs whose complement is not localizable, but that
it is NP-hard to determine if the complement of a given localizable co-well-covered graph is localizable.

\begin{theorem}\label{thm:loc-co-wc-hard}
Given a localizable co-well-covered graph $G$, it is NP-hard to determine if $\overline{G}$ is localizable.
\end{theorem}

\begin{proof}
The proof is a modification of the proof of Theorem~\ref{prop:k-localizable} specified to $k = 3$.
We make a reduction from the $3$-{\sc Colorability} problem in triangle-free graphs, which is NP-hard~\cite{MR2291884}.
Given a triangle-free graph $G$, we will construct from $G$ a localizable co-well-covered graph $G'$ such that
$G$ is $3$-colorable if and only if $\overline{G}$ is localizable.

The graph $G'$ is obtained from $G$ in two steps.
First, every edge of $G$ is extended into a triangle with a unique new vertex; let $G_1$ be the resulting graph.
Second, to every vertex $v$ of $G_1$ two new vertices are added, say $v'$ and $v''$, which are adjacent to each other and to $v$.
Then, $G'$ is the so obtained graph.
An example of the reduction is shown in Figure~\ref{fig:hardness}.

\begin{figure}[!ht]
  \begin{center}
\includegraphics[width=0.8\linewidth]{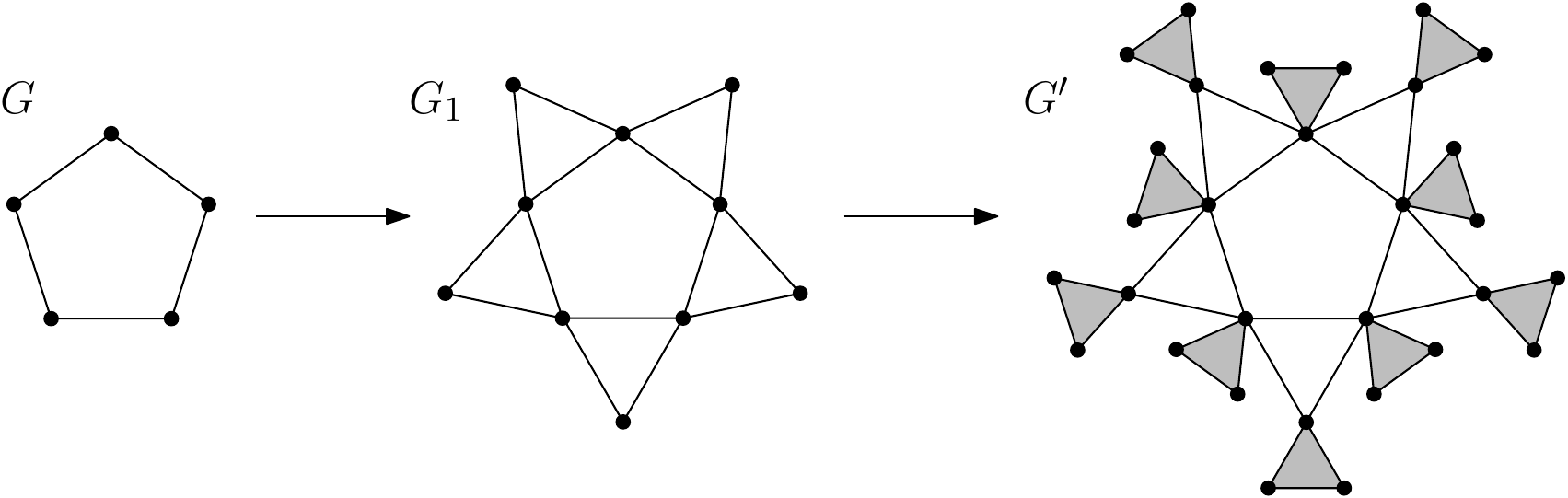}
  \end{center}
\caption{An example of the transformation $G\to G'$. The simplicial cliques  of $G'$ (which partition $V(G')$) are shaded grey.} \label{fig:hardness}
\end{figure}

Since there is a set of $|V(G_1)|$ simplicial cliques partitioning $V(G')$, the graph $G'$ is localizable.
Since $G$ is triangle-free, every maximal clique of $G'$ is of size $3$, therefore $G'$ is co-well-covered.
To complete the proof, we show that $G$ is $3$-colorable if and only if the complement of $G'$ is localizable.
Suppose that $G$ is $3$-colorable, and let $c$ be a $3$-coloring of $G$.
Extend $c$ to a $3$-coloring $c'$ of $G'$. The color classes of $c'$ define a partition of $V(G')$ into three independent sets, equivalently, a partition of
the vertex set of the complement of $G'$ into three cliques, say $C_1, C_2, C_3$.
Since all maximal cliques of $G'$ are of size $3$, every color class contains a vertex
of each maximal clique of $G$, which means that each $C_i$ is a strong clique in the complement of $G'$. Therefore
the complement of $G'$ is localizable. Conversely, since $\alpha(\overline{G'}) = \omega(G') = 3$, if the complement of $G'$ is localizable, then
 it has a partition of its vertex into $3$ strong cliques, and therefore $G'$ is $3$-colorable. But then so is $G$, as an induced subgraph of $G'$.
\end{proof}

It follows from the proof of Theorem~\ref{thm:loc-co-wc-hard} that if $G$ is any triangle-free graph of chromatic number at least $4$,
then the graph $G'$ obtained from $G$ as in the above proof (cf.~Figure~\ref{fig:hardness}) is a localizable co-well-covered graph with a non-localizable complement. There are many known constructions of triangle-free graphs of high chromatic number, see, e.g.,~\cite{MR0069494,MR0051516,MR3540616}.

We would also like to point out that in~\cite{MR3356635} Zaare-Nahandi claimed that the problem of determining if a given semi-perfect graph is well-covered is a polynomially solvable task. However, a polynomial time recognition algorithm for testing if a given semi-perfect graph is well-covered would imply P = NP. This is because of Theorem~\ref{thm:wc-rec-NP-hard} (stating that the recognition of well-covered weakly chordal graphs is co-NP-complete)
and the fact that every weakly chordal graph is perfect and hence semi-perfect.\footnote{In fact, there is an error in the arguments from~\cite{MR3356635} leading to the conclusion that the problem of determining if a given semi-perfect graph is well-covered is a polynomially solvable task.
The author assumes that a semi-perfect graph $G$ is given together with a set of $k = \alpha(G)$ pairwise disjoint cliques $Q_1,\ldots, Q_k$ with
$Q_1\cup \ldots\cup Q_k = V(G)$. The author writes:\\
{\it ``$\ldots$ Therefore, checking well-coveredness of the graph $G$ is equivalent to checking that, for each $i$, $1\le i\le k$, the set of vertices of $Q_i$ is part of a minimal vertex cover of $G$. But, this is a simple task: it is enough to check that the set of vertices of $Q_i$ is a minimal vertex cover of the subgraph of $G$ induced by $N(Q_i)$, which can be done in polynomial time.''}\\
The first sentence above is correct, in the sense that $G$ is not well-covered if and only if some $Q_i$ is part of a minimal vertex cover of $G$.
It is also true that the condition that $Q_i$ is a minimal vertex cover of the subgraph of $G$ induced by $N(Q_i)$ (assuming that $N(Q_i)$ denotes the set of vertices that are either in $Q_i$ or have a neighbor in $Q_i$)
can be tested in polynomial time. However, this condition is easily seen not to be equivalent to the condition that $Q_i$ is part of a minimal vertex cover of $G$.
This latter condition is equivalent to testing if $Q_i$ is disjoint from some maximal independent set, which, in general, is an NP-complete problem~\cite{MR1344757}.}

\section{Concluding remarks}\label{sec:conclusion}

In this paper we have initiated a study of localizable graphs, which form a rich subclass of the class of well-covered graphs.
The two properties coincide in the class of perfect graphs, and more generally for graphs in which the clique cover number
equals the independence number. We identified several classes of graphs where the property is hard to recognize and gave
efficiently testable characterizations of localizable graphs within the classes of line graphs, triangle-free graphs, $C_4$-free graphs, and cubic graphs.
Based on properties of localizable graphs, we disproved a conjecture due to Zaare-Nahandi about $k$-partite well-covered graphs having all maximal cliques of size $k$.

Our work leaves open many questions related to localizable graphs. For example, is it polynomial to check whether a given planar graph is localizable? (The corresponding question for well-covered graphs was asked in 1994 by Dean and Zito~\cite{MR1264476} and seems to be open.) Testing localizability is polynomial for triangle-free graphs. What is the complexity of the problem for $K_4$-free graphs (and, more generally, for graphs of bounded clique number)? Is it polynomial to check whether a given comparability graph is localizable? (Equivalently, given a partially ordered set, is it polynomial to test whether all its maximal antichains are of the same size?) It follows from Theorem~\ref{thm:loc-co-wc-hard} that it is NP-hard to recognize graphs such that both $G$ and its complement are localizable. What is the complexity of recognizing well-covered co-well-covered graphs?

\subsection*{Acknowledgements}

The authors are grateful to Matja\v{z} Krnc for helpful discussions, to EunJung Kim and Christophe Picouleau for feedback on an earlier draft, and to Rashid Zaare-Nahandi for providing them with a copy of paper~\cite{MR3356635}.
The work for this paper was done in the framework of a bilateral project between France and Slovenia, financed partially by the Slovenian Research Agency (BI-FR/$15$--$16$--PROTEUS--$003$). The work of A.~H.~is supported in part by the Slovenian Research Agency (research program P1-0285 and research projects N1-0032, N1-0038, J1-7051). The work of M.~M.~is supported in part by the Slovenian Research Agency (I$0$-$0035$, research program P$1$-$0285$, research projects N$1$-$0032$, J$1$-$5433$, J$1$-$6720$, J$1$-$6743$, and J$1$-$7051$).

\end{document}